\documentclass[11pt]{article}

\usepackage[letterpaper,top=2cm,bottom=2cm,left=3cm,right=3cm,marginparwidth=1.75cm]{geometry}

\usepackage{natbib}
\usepackage{hyperref}
\usepackage{url}
\usepackage{graphicx}

\usepackage{url}            % simple URL typesetting
\usepackage{booktabs}       % professional-quality tables
\usepackage{amsfonts}       % blackboard math symbols
\usepackage{nicefrac}       % compact symbols for 1/2, etc.
\usepackage{microtype}      % microtypography
\usepackage[dvipsnames]{xcolor}

\usepackage{amsmath}
\usepackage{amssymb}
\usepackage{mathtools}
\usepackage{multirow}
\usepackage{array}
\usepackage{bm}
\newcolumntype{?}{!{\vrule width 0.35mm}}

\usepackage{tikz}
\usepackage[bbgreekl]{mathbbol}

\usepackage[inline, shortlabels]{enumitem}
\usepackage{mathrsfs}

\usepackage{arash_macros}
\usepackage[small]{caption}
\usepackage{subcaption}
\usepackage{xspace}
\usepackage{verbatim}

\newcommand{\zh}{\widehat z}
\newcommand{\muh}{\widehat \mu}
\newcommand{\fh}{\widehat f}
\newcommand\Fc{\mathcal F}
\newcommand\lbd{\delta_{\rm lbl}}
\newcommand\altmin{\textsc{AltMin}\xspace}
\newcommand{\rhomin}{\rho_{\min}}
\newcommand{\Hb}{\mathbb{H}}
\newcommand{\Xc}{\mathcal{X}}

\newcommand{\T}{\mathsf{T}}

\definecolor{color1}{rgb}{0.10588235, 0.61960784, 0.46666667}
\definecolor{color2}{rgb}{0.85098039, 0.37254902, 0.00784314}
\definecolor{color3}{rgb}{0.45882353, 0.43921569, 0.70196078}
\definecolor{color4}{rgb}{0.90588235, 0.16078431, 0.54117647}

\begin{document}

\title{Step and Smooth Decompositions as Topological Clustering} %for Composite Signals}

\author{Luciano Vinas\footnote{lucianovinas@g.ucla.edu} \;\text{and}\; Arash A.~Amini\footnote{aaamini@stat.ucla.edu} \\[3ex]
{\small Department of Statistics}\\
{\small Univeristy of California, Los Angeles}\\[2ex]
%Los Angeles, CA 90095, USA \\
%{\small \texttt{lucianovinas@g.ucla.edu}}, \\
%{\small \texttt{aaamini@stat.ucla.edu}}
}

\maketitle

\begin{abstract}%   <- trailing '%' for backward compatibility of .sty file
We investigate a class of recovery problems for which observations are a noisy combination of continuous and step functions. These problems can be seen as non-injective instances of non-linear ICA with direct applications to image decontamination for magnetic resonance imaging. Alternately, the problem can be viewed as clustering in the presence of structured (smooth) contaminant. % a smooth field contaminant.
We show that a global topological property (graph connectivity) interacts with a local property (the degree of smoothness of the continuous component) to determine conditions under which the components are identifiable.
%We highlight key topological properties of the data which influence the signal recovery process and provide theoretical guarantees for the general case of a uniformly continuous signal contaminant.
%{\color{red} These scalar properties capture global information which guarantee perfect cluster recovery for qualifying, finite sample, scenarios.} 
Additionally, a practical estimation algorithm is provided for the case when the contaminant lies in a reproducing kernel Hilbert space of continuous functions. Algorithm effectiveness is demonstrated through a series of simulations and real-world studies.
\end{abstract}

\section{Introduction}
The prototypical recovery problem is nonparametric regression where we observe an unknown function %is 
corrupted by additive white noise:
%That is, we observe 
$y_i = f^*(x_i) + \varepsilon_i$, for $i=1,\dots,n$, %for
where $f^*$ belongs to some function class $\mathcal{F}$ and $\varepsilon_i$ is %as uncorrelated, random noise. 
the measurement noise.
Important to the recovery is the structure of $\mathcal{F}$ and how it can be leveraged to differentiate observations from noise. Examples of previously explored structures in nonparametric regression include: smoothness~\citep{tsybakov2009nonparametric}, sparsity~\citep{Wainwright09, bickel2009simultaneous}, homogeneity~\citep{Ke13}, and piecewise 
%simple
simplicity~\citep{Kim09,Tibshirani14}. In each of these problems, there is a particular interest in uncovering the structure-specific recovery conditions under which a finite-sample, data-estimate $\widehat{f}$ eventually recovers the optimal, data-generating $f^*$.

Another flavor of recovery problems include decompositions of the form 
\begin{equation}\label{eq:decomp}
y_i = f^*(x_i) + g^*(x_i) + \varepsilon_i,    
\end{equation}
where the recovery quantities of interest include both $f^*$ and $g^*$. Naturally, this type of recovery problem, with its multiple recoverable quantities, is more difficult than basic nonparametric regression. %previous prototypical examples. 
Examples of such decompositions %which also discover 
with provable recovery guarantees
%appropriate recovery conditions 
are rare but some notable examples include 
%results related to 
the case of sparse plus low-rank matrix recovery~\citep{Chandrasekaran09,Bahmani15,Tanner23} and %other results related to 
compressed sensing %and 
in a pair of
orthogonal bases~\citep{Donoho13}.

In this paper, we consider a nonparametric decomposition 
of the form~\eqref{eq:decomp} where the signal is
%problem where observations $y_i$ which are generated from the 
a combination of continuous and step %step-wise 
functions. We provide identifiability conditions for the continuous and %step-wise 
step functions $f^*$ and $g^*$ %are given 
in terms of the modulus of continuity of $f^*$ and the height between steps in $g^*$. Analysis of $f^*$ and $g^*$ will be sufficiently general, where each function is considered to be a mapping from a metric space $(\mathcal{X},\,d)$ to a normed vector space $(\mathcal{Y},\,\norm{\cdot})$.

In its simplest formulation, we consider $f^*$ to be real-valued and continuous, lying in a Hilbert-norm $R$-ball of a reproducing kernel Hilbert space (RKHS). For this scenario, a practical estimation algorithm is proposed with consistency guarantees given in terms of spectral quantities related to the observed kernel matrix of the RKHS. 

As in most regression analysis, we conduct our analysis under finite sampling constraints. For $g^*$ which attains at most $M$ unique values within a given sample, the composite observations will be re-expressed as
\begin{equation}\label{eq:dgp}
    y_i = f^*(x_i) + \mu^*_{z_i^*} + \varepsilon_i,\quad\text{for}\;i =1 ,\ldots,n
\end{equation}
where $\bm \mu^* \in \mathbb{R}^M$ is a vector of values referred to as the levels of $g^*$, and $z_i^*\in[M]$ are labels to the corresponding levels of $g^*$. Our main goal is to recover the labels $z_i^*$ correctly, with a  secondary goal of recovering the levels $\bm \mu^*$ and the continuous function $f^*$. 
%We assume $f^*$ to belong to class of functions $\mathcal F(\mathcal)$ on the metric space $(\mathcal X, d)$. 
For our finite sample setting, recovery of $f^* \in \mathcal{F}$ will be relaxed to finding an element of the equivalence class
\begin{equation}
    [f^*]_n = \bigl\{f\in\mathcal{F}:\;f(x_i) = f^*(x_i),\; \forall i\in[n]\bigr\}.
\end{equation}
This recovery condition may be refined to instead selecting a representative solution from $[f^*]_n$, such as a minimum-norm solution.
%$\bar{f}$, from the equivalence class $[f^*]_n$. 
An approach of this sort will depend on the regularity available in the function space $\mathcal F$
%$\mathcal{F}(\mathcal{X})$ 
and will not be a topic of focus in our forthcoming analysis.

\subsection{Applications}

To motivate the problem, let us give some concrete applications of the step and smooth decomposition model~\eqref{eq:dgp}.

\subsubsection*{Decompositions in Non-linear ICA}\label{sec:mixing}

Non-linear independent component analysis (ICA)~\citep{Hyvarinen99} provides a general framework to describe signal mixing problems. In non-linear ICA, the mixed observation $\bm y = \psi(\bm s)$ is generated using independent, latent sources $\bm s\in\reals^n$ and a non-linear, mixing function $\psi:\reals^n\to\reals^n$. In other ICA formulations~\citep{Hyvarinen19}, joint independence of $\bm s$ is relaxed to a conditional independence given some auxiliary information $\bm u\in\reals^n$. That is,
\begin{equation*}
    \log p(\bm s| \bm u) = \sum_{i=1}^n q_i(s_i, \bm u)
\end{equation*}
for appropriately defined densities $q_i$. 

Decomposition~\eqref{eq:dgp} can be understood in terms of a self-mixing, non-linear ICA problem. In the simplest scenario, we may consider sources $s_i = (x_i, u_i)$ with auxiliary information $u_i\sim {\rm Unif}[0,1)$ and mixing defined by
\begin{equation}\label{eq:mix}
    \psi(s_i) = f^*(x_i) + \mu^*_{\phi(x_i,u_i)}\quad\text{where}\quad \phi(x_i,u_i) = \lfloor u_i\cdot M \rfloor + 1.
\end{equation}
Generalizations to~\eqref{eq:mix} may consider different cut-off functions $\phi(x_i,u_i)$ which also incorporate sample spatial information $x_i$ in their cut-offs.

In contrast to traditional ICA problems, the mixing function defined in~\eqref{eq:mix} is not necessarily injective on $\reals^n$ for all choices of $f^*$ and $\phi$. This a recovery setting not covered in recent non-linear ICA literature~\citep{Hyvarinen19,Khemakhem20,Zheng2022} and one we are interested in exploring in this paper. In particular, when given partial information $\{(x_i, y_i)\}_i$, which properties of the data, if any at all, can help overcome the non-injectivity of a general $f^*$ and $\mu_\phi^*$?

\subsubsection*{Decompositions in Medical Image Correction}

In magnetic resonance imaging (MRI), image quality can be affected by factors ranging from radiofrequency coil setup to patient positioning and geometry~\citep{Asher10}. Dependent on these factors, MRI images may be contaminated with a spatially smooth, multiplicative field, known as the bias field. Figure~\ref{fig:bias} illustrates an example of a contaminated MRI image.
%Examples of a contaminated MRI image can be found in Figure~\ref{fig:bias}.

\begin{figure}[t]
    \centering
    \includegraphics[width=0.65\linewidth]{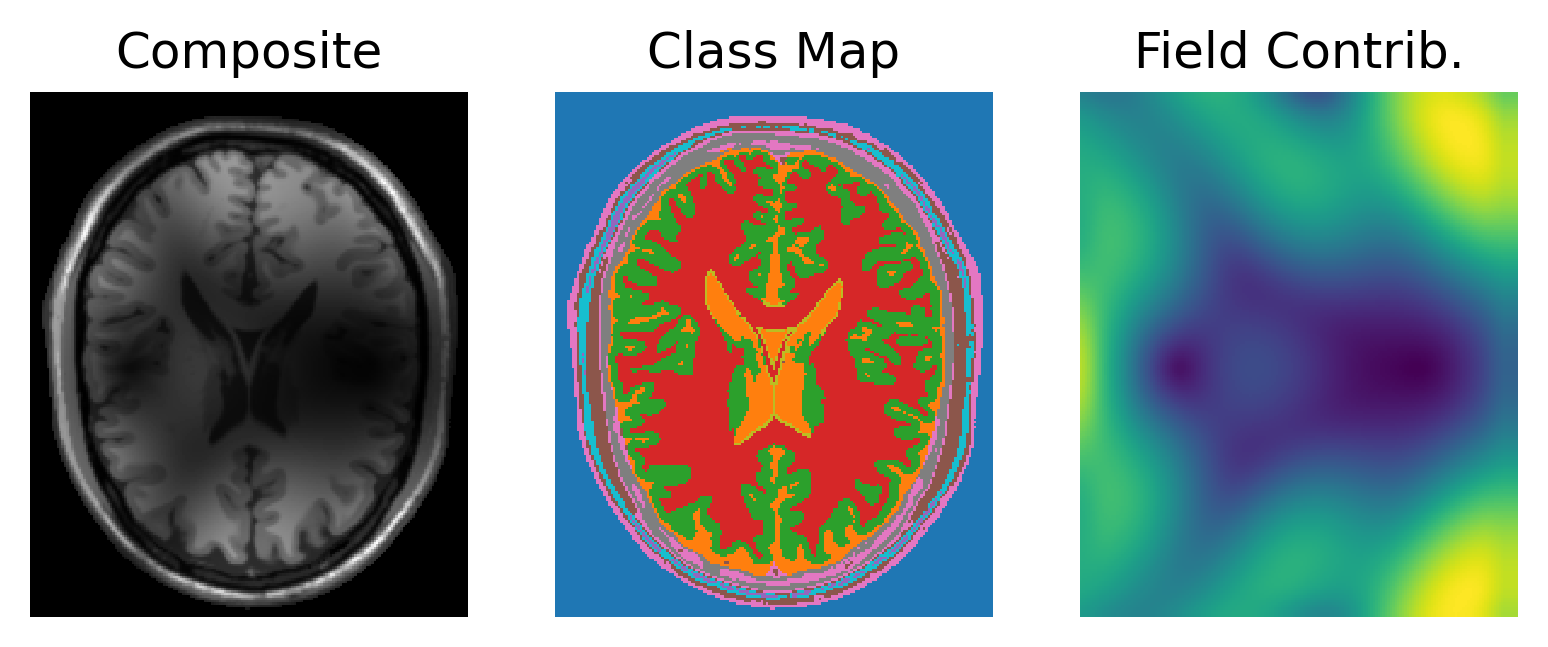}
    \caption{Example of a prominent bias field modifying the BrainWeb~\cite{Cocosco97} phantom.}
    \label{fig:bias}
\end{figure}

The MRI bias field problem admits the following multiplicative formulation~\citep{Vovk07},
\begin{equation}\label{eq:mri_bias}
    y(x) = f^*(x)\cdot \mu^*(x),\quad\text{for}\;x\in\mathcal{X}
\end{equation}
where $f^*$ is a positive smooth field on $\mathcal{X}$, and $\mu^*(x)$ are, by convention, positive tissue values at locations $x\in\mathcal{X}$. Given a fixed number of tissues classes $M$, process \eqref{eq:mri_bias} can be reformulated as \eqref{eq:dgp} under a $\log$-transformation.

% In terms of supervised learning tasks, the visual inhomgenieties produced by MRI bias fields pose considerable issue, since ground truth signal information cannot be obtained from the patient scans. Similar to the previous discussed non-linear ICA problem, this becomes an issue of learning with partial information and injectivity considerations.
In supervised learning tasks, the visual inconsistencies caused by MRI bias fields present significant challenges, as they prevent the acquisition of accurate ground truth signal information from patient scans. This issue parallels the earlier discussed problem of non-linear ICA, where, again, learning is hampered due to partial information and concerns regarding injectivity.

\subsection{Prior Work}

To the authors' best knowledge, the closest work on the theory of continuous and step decompositions is~\cite{Kim14}, %. In their paper,~\cite{Kim14} 
where they provide a characterization of the set of viable functions given an observed composite signal $h^*$. The composite $h^* = f^*\cdot g^*$ is assumed to be the product of a positive continuous function $f^*$ and a positive step-wise function $g^*$. %Using 
Assuming knowledge of the tissue ratios $\{\mu_k/\mu_{k+1}\}_{k=1}^{M-1}$,~\cite{Kim14} have shown that one %can construct a set of 
there are scalars $\{a_k\}_{k=1}^M$ such that the set
\begin{equation*}%\label{eq:scalar_mult_set}
    \widetilde{\mathcal{F}} = \bigl\{f:\mathcal{X}\to\mathbb{R}:\forall x\in \mathcal{X},\,f(x)\in \{a_k h^*(x)\}_{k=1}^M\bigr\}
\end{equation*}
contains a unique scalar multiple of $f^*$. This result is then followed by a practical algorithm which optimizes over a soft-label surrogate of $\widetilde{\mathcal{F}}$.

The theoretical result of~\cite{Kim14} is %particularly fascinating 
interesting since it dramatically reduces the search space for a viable $f$, esp. when $\mathcal X$ is finite.
%In plain terms, it states that by applying level ratio information of $g^*$ to the observations $h^*$, one can restrict the search space for $f^*$ from the set of all continuous functions to a set which is isomorphic to $[M]^{|\mathcal{X}|}$, where this second space is finite whenever $\mathcal{X}$ is finite. 
What this result does not tell us is how to identify $f^*$ in the set $\widetilde{\mathcal{F}}$, and whether $f^*$ is identifiable at all. This issue becomes readily apparent in finite sample scenarios, where %non-injectivity once again becomes a problem and 
there may be multiple ways to construct observations %from 
$h^*$ %using a
from different smooth-and-step pairs $(\tilde{f}, \tilde{g})$.
%other functions $\tilde{f}$ and $\tilde{g}$ which are continuous and step respectively. 
In short, the work of \cite{Kim14} does not address the question of identifiability which is a %the 
focus of our work. Moreover, when no level information is available, $\widetilde{\mathcal{F}}$ itself is unknown. % and the knowledge that $f^*$ belongs to $\widetilde{\mathcal{F}}$ is not very helpful.
In this regime, attempts to approximate the set $\widetilde{\mathcal{F}}$ would ultimately be sensitive to initialization choice for scale parameters $\{a_k\}_k$.

%\subsection{Our contributions}

%In this paper, we make the following contributions to the step and smooth decomposition problem:
%\begin{enumerate}
    %\setlength\itemsep{0ex}
%    \item Propose a simple alternating minimization procedure for cluster recovery, that can be implemented easily in practice for contaminants which lie in an RKHS.
%    \item Perform a one-step analysis on the proposed procedure, bounding key quantities in terms spectral tail properties, and ultimately leading to a consistency result for the contaminant estimator. 
%    \item Identify conditions under which perfect classification can be achieved in the finite sample setting and provide a deviation bound on the recovered levels in terms of a novel distance defined on the sample data.
%\end{enumerate}
%Special to the last contribution, are conditions which prove identifiability for a large class of step and smooth decomposition problems.

\section{Identifiability Theory}\label{sec:recovery}
Let $\mathcal{F}_\omega(\mathcal{X})$ be a vector space of real-valued, uniformly continuous functions with \textit{modulus of continuity} $\omega:[0,\infty)\to[0,\infty)$ over the metric space $(\mathcal{X}, d)$. That is to say, 
\begin{equation}\label{eq:Fc:omega:def}
    \Fc_\omega(\Xc) = \bigl\{f:\mathcal{X}\to\mathcal{Y}:\; \norm{f(x)-f(x')}\leq \omega(d(x,x')),\;\forall x,x'\in\mathcal{X}\bigr\}.
\end{equation}
%$\Fc_\omega(\Xc)$ is the set of functions $f$ which satisfy $|f(x)-f(x')| \leq \omega(d(x,x'))$ for all $x,x'\in \Xc$. 
%assume $f^*$ and $g^*$ have outputs in $(\reals,\, |\cdot|)$
The results of this section will be presented for outputs in $(\reals,\, |\cdot|)$, but readers interested in a more general formulation may refer to Appendix~\ref{app:ident:proofs}.
%the appendix of the supplementary material.

We consider the problem of identifying components $(f^*, \bm \mu^*, \bm z^*)$ from observations
\begin{equation}\label{eq:noiseless}
    y_i = f^*(x_i) + \mu^*_{z_i^*},\quad\text{for}\;i =1 ,\ldots,n,
\end{equation}
assuming that $f^*$ is an element of the function space $\mathcal{F}_\omega(\mathcal{X})$.
Given the commutative nature of addition, correctly identifying $(\bm \mu^*, \bm z^*)$ will automatically identify an element of $[f^*]_n$. In order to recover the true triplet $(f^*, \bm \mu^*, \bm z^*)$ we consider solving the optimization
\begin{equation}\label{opt:recovery}
    (\widehat f, \bm{\widehat \mu}, \bm{\widehat z}) = \argmin_{\substack{f \in \Fc_\omega(\mathcal{X}),\\ \bm \mu \in \mathbb{R}^M,\;\bm z\in [M]^n}}\; \frac{1}{n}\sum_{i=1}^n(y_i - \mu_{z_i} - f(x_i))^2,
\end{equation}
and its \emph{zero-mean} version where a constraint is added ensuring that $f$ is empirically zero-mean, i.e., $\sum_{i=1}^n f(x_i) = 0$. This zero-mean constraint addresses issues analogous to the scalar multiple problem described by~\cite{Kim14}. %Finally, 
By providing conditions under which~\eqref{opt:recovery} unambiguously recovers the sampled clusters $\{\mu_{z_i^*}^*\}_i$, we will have shown identifiability for the step and smooth decomposition~\eqref{eq:noiseless}.

We start by recalling the definition of a $\rho$-neighbor graph associated with a point cloud $X = \{x_i\}_{i=1}^n$ lying in a metric space $(\mathcal{X}, d)$:
\begin{defn}[Neighbor Graph]
The $\rho$-neighbor graph %$G_\rho(X) = (V,E)$ 
$G_\rho(X)$ of point cloud $X = \{x_i\}$ is the graph with vertex set %Define the $\rho$-neighbor graph $G_\rho(X) = (V,E)$ by 
%$V = [n]$ 
$[n]$
and edge set
\[
%E = 
\bigl\{(i,j)\in [n]^2:\, i\neq j\;{\rm and}\;d(x_i,x_j)\leq \rho\bigr\}.
\]
\end{defn}
The $\rho$-neighbor graph captures some aspect of the topology of the point cloud. %This neighbor graph, 
%Paired with the modulus of continuity $\omega$, this graph gives us an idea of how much a particular $f\in\Fc_\omega(\Xc)$ may vary between any two nodes $i,j \in [n]$. 
Paired with the modulus of continuity $\omega$, this graph allows us to quantify long-range variation of a particular $f^* \in\Fc_\omega(\Xc)$ via its local variations along the edges. 
%may vary between any two nodes $i,j \in [n]$
%
Control over $\Fc_\omega(\Xc)$, and by extension $f^*$, will rely on the connectedness of $G_\rho(X)$ and the size of the neighborhood distance $\rho$. Therefore, to every point cloud $X$, we associate a connectivity parameter: 
\begin{defn}[Connectivity]
For a point cloud $X$, the connectivity is defined as
\begin{align*}
    \rhomin(X) := \inf\bigl\{ \rho > 0: \; \text{$G_\rho(X)$ is connected} \bigr\}.
\end{align*}
\end{defn}

\begin{figure}[t]
    \centering
    %\includegraphics[width=0.8\linewidth]{figs/nonuniform_ex.png}
    %\caption{Contrasting example of the $\rho$-connectivity and label distance $\lbd$. The color is used to distinguish the two different clusters.}
    \begin{subfigure}{0.49\linewidth}
        \centering
        \includegraphics[width=0.95\linewidth]{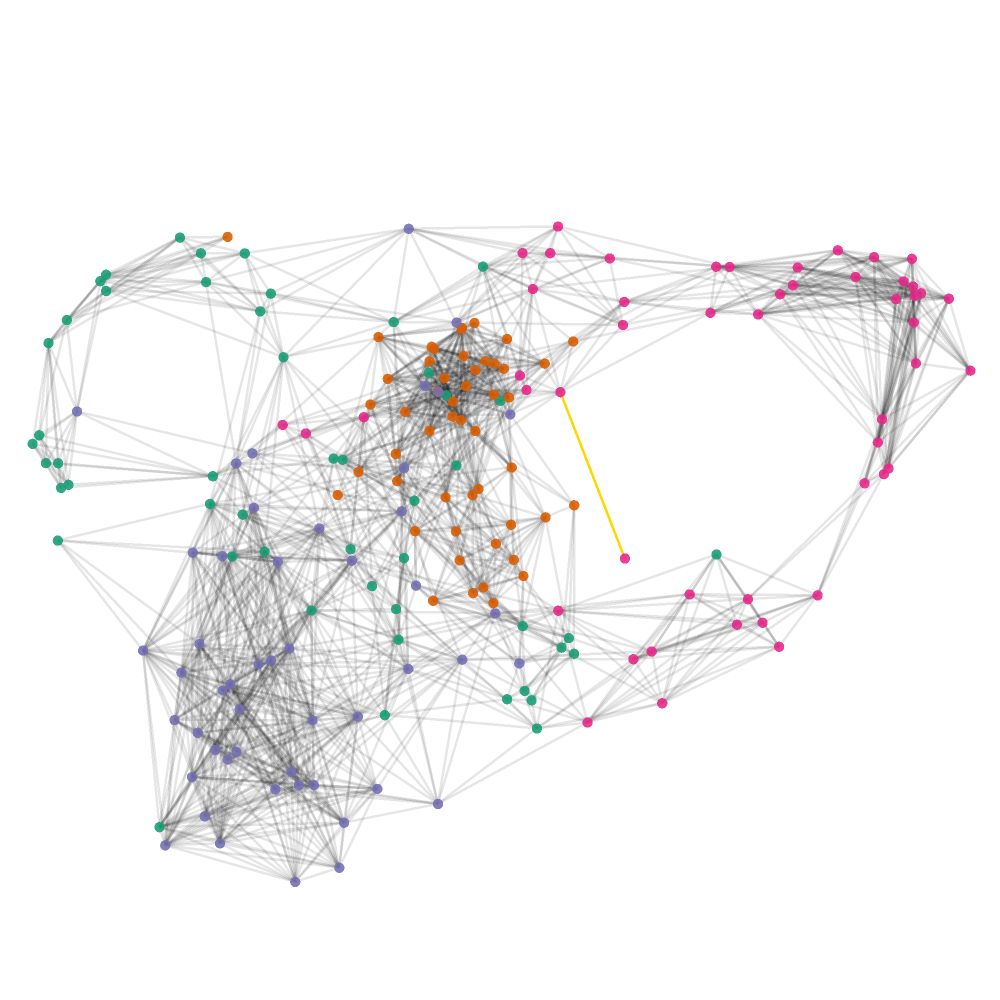}
        \caption{$\rho$-neighbor connectivity graph}
    \end{subfigure}
    \begin{subfigure}{0.49\linewidth}
        \centering
        \includegraphics[width=0.95\linewidth]{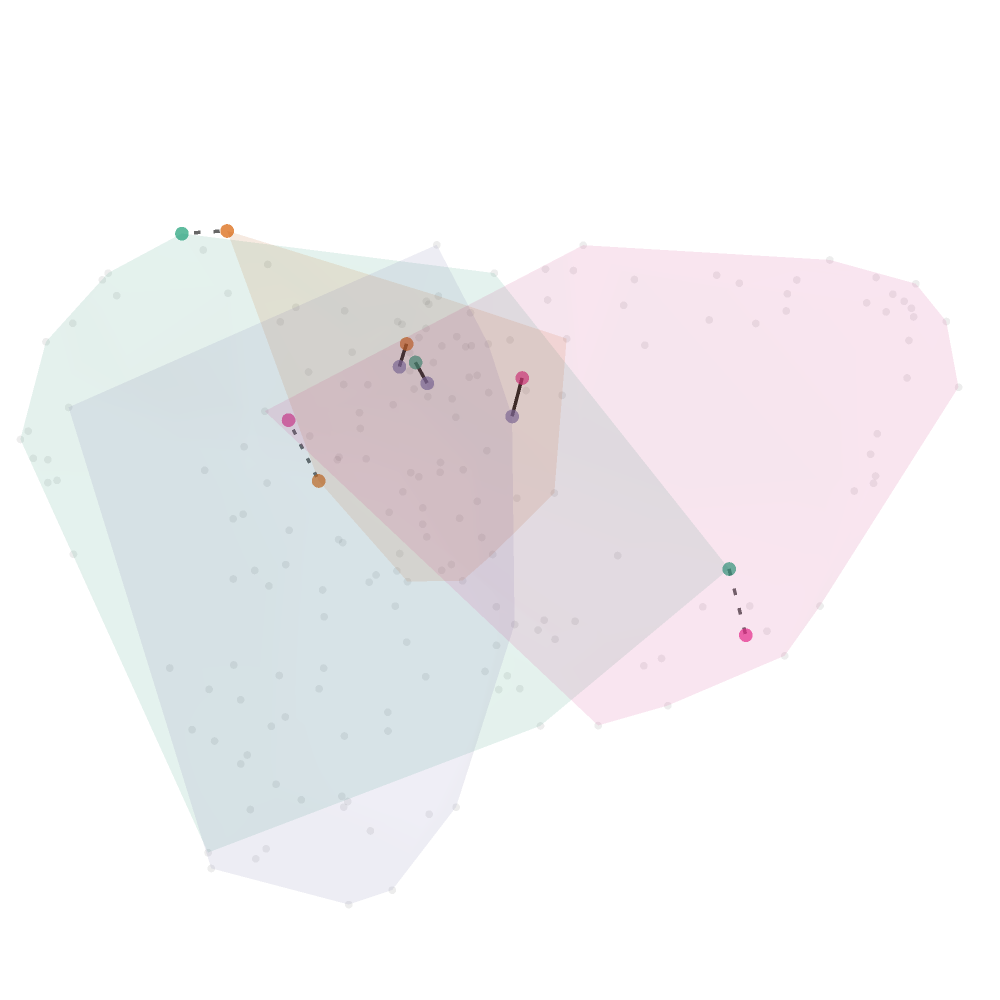} 
        \caption{Cluster connectivity graph}
    \end{subfigure}
    \caption[Caption for LOF]{A $\rho$-neighbor and cluster connectivity graph on UMAP-reduced features for four topics from the ``20 Newsgroups" classification dataset. Highlighted in gold in the left subfigure is an edge with length equal to connectivity $\rho_{{\rm min}}$.  Drawn in black in the right subfigure are the corresponding cluster distance edges $d(\mathcal{C}_k, \mathcal{C}_\ell)$. Cluster graph edges which are larger than $\delta_{\rm lbd}$ are drawn in dashed. The final cluster graph $G_{\lbd}(\mathcal{C})$ is a tree \tikz[baseline, every node/.style={shape=circle,fill=black,circle,inner sep=1.1pt,draw=black,line width=0.05pt}]{%
\node[fill=color3,yshift=0.6mm] (0) at (0,0) {};
\node[fill=color2,yshift=0.6mm] (1) at (0,0.2) {};
\node[fill=color1, yshift=0.6mm] (2) at (-0.1732, -0.1) {};
\node[fill=color4,yshift=0.6mm] (3) at (0.1732, -0.1) {};
\draw[line width=0.6pt] (0) -- (1);
\draw[line width=0.6pt] (0) -- (2);
\draw[line width=0.6pt] (0) -- (3);
} with a connecting hub at the blue colored cluster.\protect\footnotemark}
    \label{fig:nonunif_ex}
\end{figure}

\footnotetext{For interactive 3D network representation: \href{https://github.com/lucianoAvinas/topological-clustering-plots}{https://github.com/lucianoAvinas/topological-clustering-plots.}}

\nocite{McInnes18,Lang95}

\newcommand\cluster{\mathcal C}
% Just as $\rhomin(X)$ relates deviations in the smooth component $f^*$ to traversals between neighboring nodes $i,j\in [n]$ in the point cloud, we will relate
% deviations in the step component to traversals along the edges of a graph on nodes representing clusters induced by $\bm z^*$. 
% %traversals between class clusters in $\bm z^*$ to an appropriate distance quantity. 
% Each class $k\in[M]$ has an associated cluster $\cluster_k = \{i \in [n]: z^*_i = k\}$ which is part of the cluster set $\mathcal{C} = \{\cluster_k\}_{k\in [M]}$. Traversal between clusters will be measured with respect to a maximum label distance:

Similar to how $\rhomin(X)$-neighbor graph
associates deviations in the smooth component $f^*$ with traversals between neighboring nodes %$i,j \in [n]$ 
in the point cloud, we  associate deviations in the step component with traversals along the edges of a graph. This graph's nodes represent clusters induced by $\bm z^*$, namely, $\mathcal{C} = \{\cluster_k\}_{k\in [M]}$ where $\cluster_k = \{i \in [n]: z^*_i = k\}$ is the cluster corresponding to label $k$. 
%The traversals between class clusters in $\bm z^$ are related to an appropriate distance quantity.
Traversal between clusters will be measured with respect to the following \emph{cluster distance}:
%\begin{definition}[Label Distance]
%For paired data $(X,\bm z^*)$, we define a notion of \emph{cluster distance}
\begin{equation}\label{eq:pairwise_clust}
    d(\cluster_k, \cluster_\ell)  \coloneqq
    \min_{i \in \cluster_k, \, j \in \cluster_\ell} 
    d(x_i, x_j).
\end{equation}
The pairwise distances of $\mathcal{C}$, although not corresponding to a metric space, can be used with a tolerance $\delta > 0$ to construct a neighbor graph $G_\delta(\mathcal{C})$, with the edge set
\[
\bigl\{ (\cluster_k, \cluster_\ell):\; k\neq \ell \; \text{and} \;\; d(\cluster_k, \cluster_\ell) \le \delta \bigr\}.
\]
%Let $G_\delta(\mathcal{C})$ be the $\delta$-neighbor graph of $\mathcal{C}$ constructed with respect to pairwise distances~\eqref{eq:pairwise_clust}. 
\begin{defn}[Label distance]
    The label distance for paired data $(X,\bm z^*)$ is
    \begin{equation}
        \lbd(X,\bm z^*) \coloneqq \inf\bigl\{ \delta > 0: \; \text{$G_\delta(\mathcal{C})$ is connected} \bigr\}.
    \end{equation}
\end{defn}
%where $G_\delta(\mathcal{C})$ is graph over the $M$ clusters $\mathcal{C} = \{\mathcal{C}_k\}_{k=1}^M$.
%    \[
%    \lbd = \inf\bigl\{ \delta > 0: \; \text{$G_\delta(\mathcal{C})$ is connected} \bigr\}.
%    \]
%\begin{equation}
%    \delta_{k \ell}(X, \bm z^*)  \coloneqq
%    \min_{i \in \cluster_k, \, j \in \cluster_\ell} 
%    d(x_i, x_j).
%\end{equation}
%The label distance for data $(X,\bm z^*)$ is defined by
%\begin{align}
%    \lbd(X, \bm z^*)
%    \coloneqq \min_{G \in \mathcal{G}_M} \max_{(k,\ell)\in G} \delta_{k\ell},
%\end{align}
%where $\mathcal{G}_M$ is the set of all connected graphs on $[M]$.
%\end{definition}

When clear from context, dependence on sample $(X, \bm z^*)$ will be omitted from all defined terms. 
%A visualization of the different defined graph quantities is shown in Figure~\ref{fig:nonunif_ex}.
Figure~\ref{fig:nonunif_ex} shows an example of a $\rho$-neighbor graph and the associated quantities.

%If desired, the label distance can be described in a fashion similar to the connectivity coefficient, where if $\mathcal{C} = \{\mathcal{C}_k\}_{k=1}^M$ and $d(\cluster_k,\cluster_\ell) \coloneqq \min_{i \in \cluster_k, \, j \in \cluster_\ell} d(x_i, x_j)$ then
%    \[
%    \lbd = \inf\bigl\{ \delta > 0: \; \text{$G_\delta(\mathcal{C})$ is connected} \bigr\}.
%    \]
%This yields an immediate corollary of $\lbd \le \rhomin$. A visualization of the different defined graph quantities is shown in Figure~\ref{fig:nonunif_ex}.

\medskip
Our main result is the following  %lossless 
cluster recovery guarantee:
\begin{thm}[Cluster recovery]\label{thm:misclass:v1}
    Let $X = \{x_i\}_{i=1}^n$ be a point cloud in a metric space $(\mathcal{X}, d)$ and let $\{y_i\}_{i=1}^n$ follow  model~\eqref{eq:noiseless}%\eqref{eq:dgp} 
    with $f^* \in \Fc_\omega(\Xc)$ and $\bm z^* \in [M]^n$. %Suppose every label in $[M]$ is attained by $\bm z^*$ on $[n]$.
    %Assume that 
    If the connectivity $\rhomin$ of $X$ satisfies 
    \begin{align}\label{eq:cluster:recovery:condition}
        \omega(\rhomin) < \frac1{2M}\min_{k \neq \ell} |\mu_k^*-\mu_\ell^*|,
    \end{align}
    then, the labels $\bm \zh$ produced by~\eqref{opt:recovery} have zero misclassification error relative to $\bm z^*$.
\end{thm}
%This results matches what our intuition might say about the problem. It states that if the uniform spatial resolution $\rhomin$ of the data $X$ is sufficiently sharp such that its corresponding variation in $f$ is much smaller than the difference between sampled levels $\bm \mu$, then it should be simple to identify the correct labels $\bm \zh = \bm z^*$.
Our next result is an error bound on the recovered levels $\bm \muh$: % assuming for a given $f^*$:
\begin{prop}[Level recovery]\label{prop:dev_bnd}
    Under the assumptions of Theorem~\ref{thm:misclass:v1}, let $(\fh, \bm \muh, \bm \zh)$ be the solution of the zero-mean version of problem~\eqref{opt:recovery}. Then, we have
    \begin{align}\label{eq:level:dev}
        \max_{k\in [M]}| \mu_k^* - \muh_k|
        \leq 2(M-1)\, \omega(\lbd) + \Bigl|\frac{1}{n}\sum_{i=1}^n f^*(x_i)\Bigr|.
    \end{align}
\end{prop}
In essence, both Theorem~\ref{thm:misclass:v1} and Proposition \ref{prop:dev_bnd} provide deviation bounds under specific connectivity constraints. The quantities $\rhomin$ and $\lbd$ gauge the minimum jump distances at which the induced graphs of $\{x_i\}_{i=1}^n$ and $\{\mathcal{C}_k\}_{k=1}^M$ remain connected. The modulus $\omega(\cdot)$ then translates these jumps in distances into equivalent jumps in levels, observed indirectly through $\{y_i\}$.

%For some intuition, both Theorem~\ref{thm:misclass:v1} and Proposition \ref{prop:dev_bnd} have a recurring theme of deviation bounds given under connectivity constraints. Quantities $\rhomin$ and $\lbd$ measure minimum jump distances for which induced graphs of $\{x_i\}_{i=1}^n$ and $\{\mathcal{C}_k\}_{k=1}^M$ are still connected. Modulus $\omega(\cdot)$ then translates this jump in distance to an equivalent jump in levels which we observe indirectly through $\{y_i\}_i$. 

Theorem~\ref{thm:misclass:v1} says that perfect cluster recovery $\bm \zh \equiv \bm z^*$ is attainable if this translated jump is roughly below the minimum resolution of the true levels $\{\mu^*_k\}$. 
Proposition~\ref{prop:dev_bnd} has a similar theme, but in the context of level recovery, where the possible values of %Proposition~\ref{prop:dev_bnd} is similar, except that for level recovery, the possible values of 
$\mu_k^*\in\reals$ are not discretized. 
This leads to a gradual reduction in error as outlined in Proposition~\ref{prop:dev_bnd}, contrasting 
%This yields a gradual decrease in error, shown in Proposition~\ref{prop:dev_bnd}, which contrasts 
with the sharp recovery of discrete labels $z_i^* \in [M]$
%. result of 
in Theorem~\ref{thm:misclass:v1}.  

%, which deals with discrete labels $z_i^* \in [M]$.

The remainder term $|\frac1n\sum_i f^*(x_i)|$ in~\eqref{eq:level:dev}
%present in Proposition~\ref{prop:dev_bnd}
highlights the scalar-shift ambiguity inherent in the components of  model~\eqref{eq:noiseless}, 
%which affects observations $y_i$, 
where for any scalar %contribution 
$c\in\reals$ we can rewrite~\eqref{eq:noiseless} %re-expressed 
as
\[
y_i = (f^*(x_i) - c) + (\mu^*_{z_i^*} + c).
\]
In other words, the two components are only identifiable up to a scalar shift.
More generally, problem~\eqref{opt:recovery} can be extended to
include a constraint $\frac1n \sum_{i=1}^n f(x_i) = \widebar{f^*}$, 
%consider constraint sets to
for some prescpecified average value $\widebar{f^*}$, in which case
%With this new formulation 
%The same proof for 
Proposition~\ref{prop:dev_bnd} holds with the remainder term  $|\frac1n\sum_i f^*(x_i) - \widebar{f^*}|$.

As an immediate corollary to Proposition~\ref{prop:dev_bnd}, one can show that, under mild regularity on the sampling of $(X,\bm z^*)$, the zero-mean recovery problem~\eqref{opt:recovery} achieves asymptotic identifiability of $(\bm\mu^*,\bm z^*)$. As this corollary references multiple sets of samples, the notation $(\cdot)^{(n)}$ will be used to differentiate parameters belonging to different sets of observations $\{y_i\}_i$. %Lastly, 
We also allow the number of observed levels $M_n$ 
%of  unique observed values of $g^*(x)$ 
to grow with $n$. We say that a condition is \emph{eventually satisfied} if it holds for all $n \ge N$ for some $N \in \nats$.

\begin{cor}\label{corr:lim_id}
Consider a sequence of point clouds $\{X^{(n)}\}$, with corresponding true labels $\{\bm z^{*(n)}\}$ and class levels $\bm \mu^{*(n)} \in \reals^{M_n}$. Let $\lbd^{(n)}$ be the label distance for $(X^{(n)}, \bm z^{*(n)})$. Assume that the connectivity condition~\eqref{eq:cluster:recovery:condition} 
%of Theorem~\ref{thm:misclass:v1} 
is eventually satisfied, and as $n \to \infty$,
\begin{align*}
    \omega(\lbd^{(n)}) = o(M_n^{-1}), %o(M_n), 
    \quad %\text{and}\quad 
    \frac{1}{n}\sum_{x\,\in\, X^{(n)}} f^*(x) = o(1).
\end{align*}
Then for any solution $(\fh^{(n)},\bm \muh^{(n)},\bm \zh^{(n)})$ of the zero-mean version of problem~\eqref{opt:recovery}, %we have
\begin{align*}
    \lim_{n\rightarrow \infty} \max_{k\in M_n} \bigl|\mu_k^{*(n)} - \muh_k^{(n)}\bigr| = 0.
\end{align*}
\end{cor}

According to Corollary~\ref{corr:lim_id},
when $\{M_n\}$ is bounded, a set of sufficient conditions for recovery of both clusters and levels is:
\[\rhomin^{(n)} = o(1), \quad  
\lbd^{(n)} = o(1), \quad \text{and}\quad 
\Delta_n := \min_{k\neq \ell} |\mu_k^{*(n)} - \mu_\ell^{*(n)}| = \Omega(1),\]
i.e., minimum level gap is bounded below.
%Deviations conditions which depend on $M_n$ can be simplified for cases where the sequence $\{M_n\}_n$ has a tight bound $M_n \leq M$. 
When $\{M_n\}$ is unbounded, both the connectivity $\rhomin$ and the label distance $\lbd$ must decrease more rapidly. For example, when the smooth component is Lipschitz (i.e., $\omega(t) = Lt$), a set of sufficient conditions are
\[
\rhomin^{(n)} = o(\Delta_n / M_n), \quad \lbd^{(n)} = o(1/M_n).
\]
%be decreasing sufficiently rapidly in terms of $M_n$ and the minimum level gap. %difference $\min_{k\neq \ell}|\mu_k - \mu_\ell|$.
Note that Corollary~\eqref{eq:cluster:recovery:condition} is a deterministic result, but it can be translated to a high probability version given appropriate assumptions on the sampling distribution of $(X,\bm z)$.
%In either case, both results are given in terms of deterministic quantities but these may be translated to results in probability given the appropriate assumptions on the sampling distribution of $(X,\bm z)$.

\medskip
%While the identifiability results of this section are intuitive and  described in terms of easily understood topological quantities, 
The identifiability results of this section are intuitive and  described in terms of easily understood topological quantities. However,
it is noteworthy 
%it bears mentioning 
%that it is not immediately clear 
that how to
obtain a perfect classification result similar to Theorem~\ref{thm:misclass:v1} is not immediately clear.
%, using a pair-wise deviation bound $\omega(\rhomin)$. 
That is, irrespective of the placements of labels $\bm z^*$ on the point cloud $X$, and regardless of the dimension of the space carrying $X$,  
%$G_\rho(X)$, we have shown that 
we have shown that one can globally control $\bm\zh$ using only 
%local deviation techniques and additional 
a scalar %information of 
parameter of the point cloud, namely, the radius of connectivity of its associated neighbor graphs $G_\rho(X)$.

\section{Methods and Optimization}\label{sec:methods}
\newcommand\Kc{\mathcal K}

For practical estimation, we consider estimating functions $f^* \in \Hb$ lying in the Hilbert-norm $R$-ball of an RKHS. The following example shows that this case can be treated as a special case of~\eqref{eq:Fc:omega:def} with a linear modulus $\omega(t) = O(t)$.
%This parametrized class is a strict superset to the RKHS Hilbert-norm ball considered in Section~\ref{sec:methods}.

\begin{exa}\label{ex:kernel_dist}
    Consider the case where $f^*$ lies in RKHS $\Hb$. The natural metric to consider on $\Xc$ is the so-called \emph{kernel metric}
    \begin{align}
        d_{\mathcal{K}}(x,x') &:= \norm{\Kc(x,\cdot) - \Kc(x',\cdot)}_{\Hb}  
        = \sqrt{\mathcal{K}(x,x) - 2\mathcal{K}(x,x') + \mathcal{K}(x',x')}.
    \end{align}
    Using the Cauchy--Schwarz inequality, it is straightforward to show the following Lipschitz property: For any $f \in \Hb$, we have
    \[
    |f(x) - f(x')| \le \norm{f}_{\Hb} \, d_{\Kc}(x,x')
    \]
     for all $x,x'\in \Xc$. Letting $\omega_f$ denote a modulus of continuity of function $f$, the above shows that one can take $\omega_f(t) = \norm{f}_\Hb \cdot t$ for all $f \in \Hb$. If we further assume $\norm{f^*}_\Hb \le R$ for some constant $R$, then $\omega(t) = O(t)$. 
     %In the case of Section~\ref{sec:methods}, we have $\norm{f^*}_\Hb = R$.
\end{exa}

\paragraph[AltMin Algorithm]{\altmin Algorithm} For our estimation procedure, we propose a blockwise coordinate descent with alternating updates on $(\bm \mu, \bm z)$ and $f$. More specfically, in each iteration, the current estimates $(\fh, \bm \muh, \bm \zh)$ are updated to the new ones $(\fh^+,\bm \muh^+,\bm \zh^+)$ by
\begin{align}
    &\hspace{-1mm}\fh^+ = \argmin_{f \in \Hb}\frac{1}{n}\sum_{i=1}^n(y_i - \muh_{\zh_i} - f(x_i))^2 + \tau \norm{f}_{\Hb}^2,   \label{opt:rep_opt} \\
    &\hspace{-1mm}(\bm \muh^+, \bm \zh^+) = 
    \argmin_{\bm \mu \in \mathbb{R}^M,\, \bm z\in [M]^n} \frac{1}{n}\sum_{i=1}^n(y_i - \mu_{z_i} - \fh^+(x_i))^2,\label{opt:km_opt}
\end{align}
with $\tau$ and $M$ being values to be determined through a cross-validation procedure.

For fixed $\fh$, optimization~\eqref{opt:km_opt} can be solved through a $k$-means procedure. For RKHS $\Hb$ equipped with kernel $\mathcal{K}:\mathcal{X}\times\mathcal{X}\rightarrow \mathbb{R}$, optimization (\ref{opt:rep_opt}) has the following representer solution
\begin{equation}\label{eq:fhat}
    \fh^+ = \frac{1}{\sqrt{n}}\sum_{i=1}^n \widehat\alpha^+_i \Kc(x_i, \cdot), \;\; \bm{\widehat\alpha^+} := (K + \tau I_{n})^{-1}(\bm{y}-\widehat{Z}\bm \muh)/\sqrt{n}
\end{equation}
where $K$ is the $n \times n$ kernel matrix with entries $K_{ij} = \mathcal{K}(x_i,x_j)/n$ and $\widehat{Z} \in\{0,1\}^{n\times L}$ is the one-hot encoding label matrix for previous label estimate $\bm\zh$. \nocite{Wainwright_HDS}

\newcommand\gh{\widehat g}
\subsection{One-step Analysis}\label{sec:one:step:analysis}

%In this section, we take a closer look at the alternating interactions of updates~(\ref{opt:rep_opt}-\ref{opt:km_opt}). The one-step iteration bounds provided highlight the difficulty of optimizing~\ref{opt:recovery} under computational constraints and provide a proxy to understand the generalization behavior of the full \altmin algorithm.
In general, the interaction between updates~\eqref{opt:rep_opt} and~\eqref{opt:km_opt} may be quite complicated. In this section we  show a positive result: In the large sample limit, classification with \altmin simplifies to classification with regular $k$-means on the uncontaminated (step) signal.

\newcommand\ff{\breve{f}}
\newcommand\gf{\breve{g}}

%For observations $\{y_i\}_i$ generated from~\eqref{eq:decomp}, we assume $\varepsilon_i$ is i.i.d. with zero-mean and variance $\sigma^2$. For $g^*$, we consider a step-wise signal with $g^*(x_i) = \mu_{z_i^*}^*$. To get a handle on the misclassification error for the $k$-means estimate $\gh$ of $g^*$, we consider using the predictive performance of a general estimate $\fh$ of $f^*$. In particular we have the following:
%\begin{prop}\label{prop:km_err}
%    Suppose $\fh$ has predictive error
%    \[
%    \frac{1}{n} \sum_{i=1}^n (y_i - \fh(x_i) - \mu^*_{z^*_i})^2 \leq \delta^2,
%    \]
%    and $\gamma = \min_{k\neq \ell} |\mu_k^* - \mu_\ell^*|$. Let $\zh$ be the $k$-means estimate on $\{y_i - \fh(x_i)\}_i$ then
%    \[
%    {\rm Miss}(z^*,\zh) \leq \delta^2/\gamma^2.
%    \]
%\end{prop}
%Now consider a specific $\fh$ produced from update~\eqref{eq:fhat} with point evaluations $\bm \fh = (\fh(x_i))_{i=1}^n$. Our next goal will be to decompose the predictive error of $\fh$ into multiple interpretable terms.
We consider observations $\{y_i\}_i$ drawn from~\eqref{eq:decomp}~%\eqref{eq:mix} 
with i.i.d. zero-mean noise $\varepsilon_i$ of variance $\sigma^2$. 
%Similar to before,
As before, $g^*$ will be assumed to be a step signal with $g^*(x_i) = \mu_{z_i^*}^*$, although the results of this section hold for any $g^*$ that is \emph{sufficiently outside} the RKHS, as will be made precise in Corollary~\ref{corr:asym_km_err}. For our analysis, we consider a half-step of the \altmin algorithm, evaluating performance after update~\eqref{eq:fhat}. Our goal is to show the pointwise consistency of the KRR estimator $\bm\fh \coloneqq (\fh(x_i))$, that is
\begin{equation}\label{eq:point_eval}
\lim_{n\to\infty}\mathbb{E}_{\bm \varepsilon}\, {\rm MSE}(\bm f^*, \bm \fh) = 0,
\end{equation}
where ${\rm MSE}(\bm a, \bm b) = \norm{\bm a - \bm b}_2^2/n.$
%For a kernel matrix $K$ with 

Let  $K = V \Lambda V^\T$ be the eigenvalue decomposition of the kernel matrix where $\Lambda = \diag(\lambda_i, i \in[n]$), and define 
\[h(\lambda;\tau) := \frac{\lambda^2}{(\lambda + \tau)^2}, 
\quad \Gamma_\tau := \sqrt{h(\Lambda; \tau)} = \Lambda (\Lambda + \tau I)^{-1}
\]
extending a scalar function to diagonal matrices in the natural way (i.e., by applying to each diagonal entry.) 
We assume the eigenvalues are ordered as follows: $\lambda_1 \ge \lambda_2 \ge \cdots \ge \lambda_n$.  
%Standard compactness arguments guarantee that $\lambda_n \to 0$ as $n \to \infty$, i.e., the eigenvalues decay, a detail that is pertinent to the subsequent discussion.
Consider the Fourier expansion of $f^*$ and $g^*$ in the  (empirical) eigen-basis of the kernel, that is, $\bm \ff \coloneqq (\ff_i) \coloneqq V^\T \bm f^*$ and $\bm \gf \coloneqq (\gf_i) \coloneqq V^\T \bm g^*$. Then
%$\bm a = V^\T \bm f^*$ and $\bm b = V^\T \bm g^*$. 
%The expected predictive error of $\fh$ w.r.t. the noise $\bm \varepsilon$ goes as,
%\aaa{Can we put the derivation in a later section or appendix, only keep the final result:}
\begin{align}
    \frac{1}{n}\mathbb{E}_{\bm \varepsilon}\norm{\bm f^* - \bm \fh}_2^2 &= 
    \frac{1}{n}\mathbb{E}_{\bm \varepsilon}\norm{(I_n- \Gamma_\tau) V^\T \bm f^* -\Gamma_\tau V^\T\bm g^* - V^\T \bm\varepsilon }_2^2 \notag \\
    &\leq \frac{2}{n}\norm{ (I_n- \Gamma_\tau) \bm \ff}_2^2 +   \frac{2}{n}\norm{\,\Gamma_\tau \bm \gf\,}_2^2 + 
        \frac{\sigma^2}{n}{\rm tr}(\Gamma_\tau^2) \notag \\
    &= \frac{2}{n}\sum_{i=1}^n \frac{\tau^2  \ff_i^2}{(\lambda_i + \tau)^2 } 
 %+ \frac{2}{n}\norm{\Gamma_\tau \bm \gf}_2^2 
 + \frac{2}{n}\sum_{i=1}^n 
    %\frac{\lambda_i^2  \gf_i^2}{(\lambda_i + \tau)^2 } 
    h(\lambda_i; \tau) \,\gf_i^2
 + \frac{\sigma^2}{n} \sum_{i=1}^n \frac{\lambda_i^2}{(\lambda_i + \tau)^2} \label{eq:decomp:1}
\end{align}
The first and the third terms are the bias and variance, respectively, for recovering $\bm f^*$ in classical kernel ridge regression (KRR). Both can be made to go to zero as $n \to \infty$ for a proper choice of $\tau = \tau_n = o(1)$. The middle term is new to our decomposition, and is the filtering effect of KRR on the step component $g^*(\cdot)$. 

Intuitively, since a discontinuous $g^*$ is not in the RKHS, one would have 
\[
\min_{\substack{\bar{g}\in\Hb:\\ (\bar g(x_i)) = \bm g^*}} \norm{\bar g}_\Hb \to \infty\quad\text{as }n\to\infty.
\]
This in turn implies $\frac1n \sum_{i=1}^n \gf_i^2 / \lambda_i \to \infty$, 
 forcing the middle term in~\eqref{eq:decomp:1} to become negligible, implying that KKR effectively filters out $g^*$. 
%which forces the middle term in~\eqref{eq:decomp:1} to be small (i.e., KKR filters out $g^*$). 
To see this, note that since $\lambda_i$ are decaying as a function of $i$, for the expression $\frac1n \sum_{i=1}^n \gf_i^2 / \lambda_i$ to grow without bound, most of the energy of $g^*$ (where ``energy'' is defined as $\frac1n \sum_{i=1}^n  \gf_i^2$) must  be concentrated on the higher-index components, which correspond to smaller eigenvalues.
%From a signal-processing standpoint, 
Multiplication by  %$h(\lambda;\tau) = \lambda^2 / (\lambda + \tau)^2$ 
$h(\lambda_i, \tau)$   filters out components of $g^*$ associated with small eigenvalues; equivalently it acts a a low-pass filter, filtering out higher index (i.e., higher frequency) components. 
%For $\tau = o(1)$ and $n\to \infty$, the first (bias) term goes to zero while the third 
%(irreducible) term is on the order of $O(\sigma^2)$. 
%For second mis-specification term,

To make the above intuition more precise, 
%and control the middle term in~\eqref{eq:decomp:1}, 
%we define the following terms. 
consider the  \emph{spectral survival function} of $g^*$:
\begin{align}\label{eq:spec:surv}
     S_{g^*}(t):=\sum_{i=1}^{n} \frac{\gf_i^2}{n}\cdot 1\{\lambda_i > t\}.
\end{align}
As $t \to \infty$, $S(t)$ goes to zero, and the faster this decay, the more $g^*$ is concentrated on higher-index components. That is, the tail behavior of $S_{g^*}(t)$ is what determines how well $g^*$ is filtered by KRR. Let $r_n = \max \{i\in[n]:\, \gf_i^2 > 0\}$ and let $\beta_n$ be the largest $\beta \ge 0$ that satisfies
\begin{equation}\label{eq:tail_bnd}
    S_{g^*}(t)%:=\sum_{i=1}^{n} \frac{\gf_i^2}{n}\cdot 1\{\lambda_i > t\} 
    \;\leq\; \norm{\bm g^*}_\infty^2\cdot \Bigl(\frac{\lambda_{r_n}}{t} \Bigr)^{\beta}, \quad \text{ for all $t > 0$.}
\end{equation}
Such a tail bound always exists, since the trivial case $\beta = 0$ reduces to $\norm{\bm g^* /\sqrt{n}}_2^2 \leq \norm{\bm{g}^*}_\infty^2$.  %It also relates to the tail decay of a certin discrete distribution; see Appendix~?.
%\begin{rem}
%    The current description of $r_n$ simplifies analysis but could be too stringent in practice. A relaxed definition of $r_n$ may consider an auxiliary, decreasing sequence $\epsilon_n\to 0$ such that 
%    \[
%    r_n =  \max \bigl\{i\in[n]:\, \sum_{j:j< i}\gf_j^2/n \leq \epsilon_n \bigr\}.
%    \]
%    Note that this definition coincides with the previous in the case $\epsilon_n = 0$. Relative to this new definition, parameter $\beta_n$ would be defined as the bounding exponent to expression $\frac{1}{n}\sum_{i\geq r_n} \gf_i^2\cdot 1\{\lambda_i > t\}$.
%\end{rem}
The parameters of the tail bound are influenced by how much the higher-index components of $\bm \gf$ contribute to the total norm (or energy).
%From a signal-processing standpoint, multiplication with the function $h(\lambda;\tau) = \lambda^2 / (\lambda + \tau)^2$ filters out components associated with small eigenvalues.
The tail bound works together with the spectral filter $h(\lambda;\tau)$ to give the following control for the middle term of~\eqref{eq:decomp:1}:
\begin{prop}\label{lem:beta_lem}
    %Let $\beta_n$ be the largest $\beta$ which satisfies~\eqref{eq:tail_bnd}. 
    %Assume that $\lambda_{r_n}\leq \tau_n \leq \min\{1,\lambda_1\}$, and
    Consider KRR with regularization parameter $\tau_n$
    and 
    let $\xi_n := \lambda_{r_n} / \tau_n$. %and $r_n,\beta_n$ to be defined as before.
    Then, %for some $c_n = c(\beta_n) < 1$,
    \begin{equation}\label{eq:gamma_decay}
        %\frac{1}{n}\norm{\Gamma_\tau \bm \gf}_2^2
        \frac{1}{n}\sum_{i=1}^n 
        h(\lambda_i; \tau_n) \,\gf_i^2
        \;\lesssim\; 
        \max\{ \xi_n^2,\, \xi_n^{\beta_n}\},
        %\max\Bigl\{\Bigl(\frac{\lambda_{r_n}}{\tau}\Bigr)^2,\, c_{\beta_n} \Bigl(\frac{\lambda_{r_n}}{\tau} \Bigr)^{\beta_n} \Bigr\}
    \end{equation}
    where $\lesssim$ denotes inequality up to universal constants.
\end{prop}
We note that the best case scenario in Proposition~\ref{lem:beta_lem} is obtained when $r_n = n$ and $\beta_n \ge 2$, leading to the quickest possible decay of $O(\xi_n^2) = O((\lambda_n / \tau_n)^2)$ for the residual norm.

\begin{figure}[t]
    \begin{subfigure}{0.49\linewidth}
        \centering
        \includegraphics[width=\linewidth]{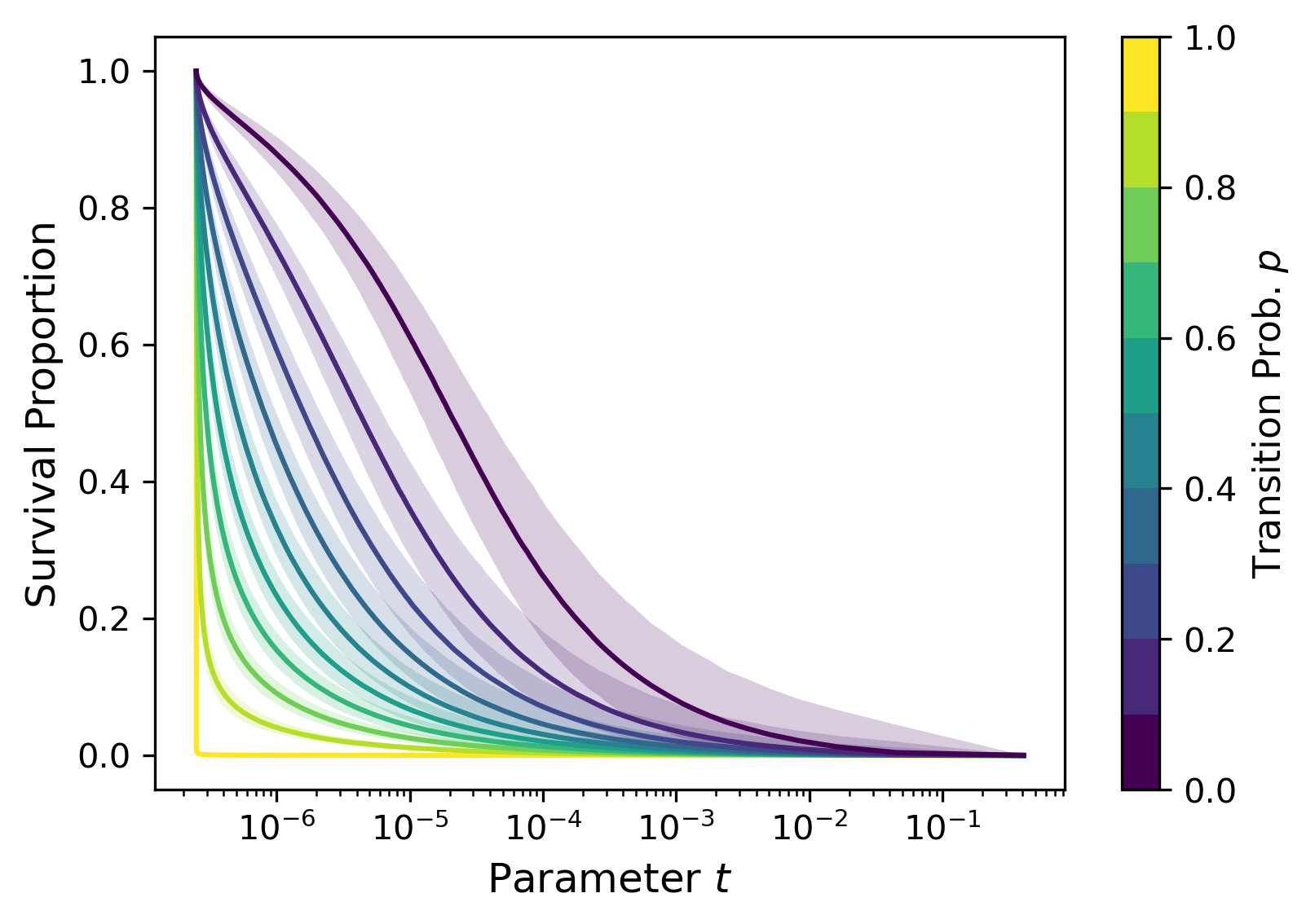}
        \caption{Survival functions of $\bm \gf$ for length $n=1000$.}
        \label{fig:markov:cdfs}
    \end{subfigure}
    \hfill
    \begin{subfigure}{0.49\linewidth}
        \centering
        \includegraphics[width=\linewidth]{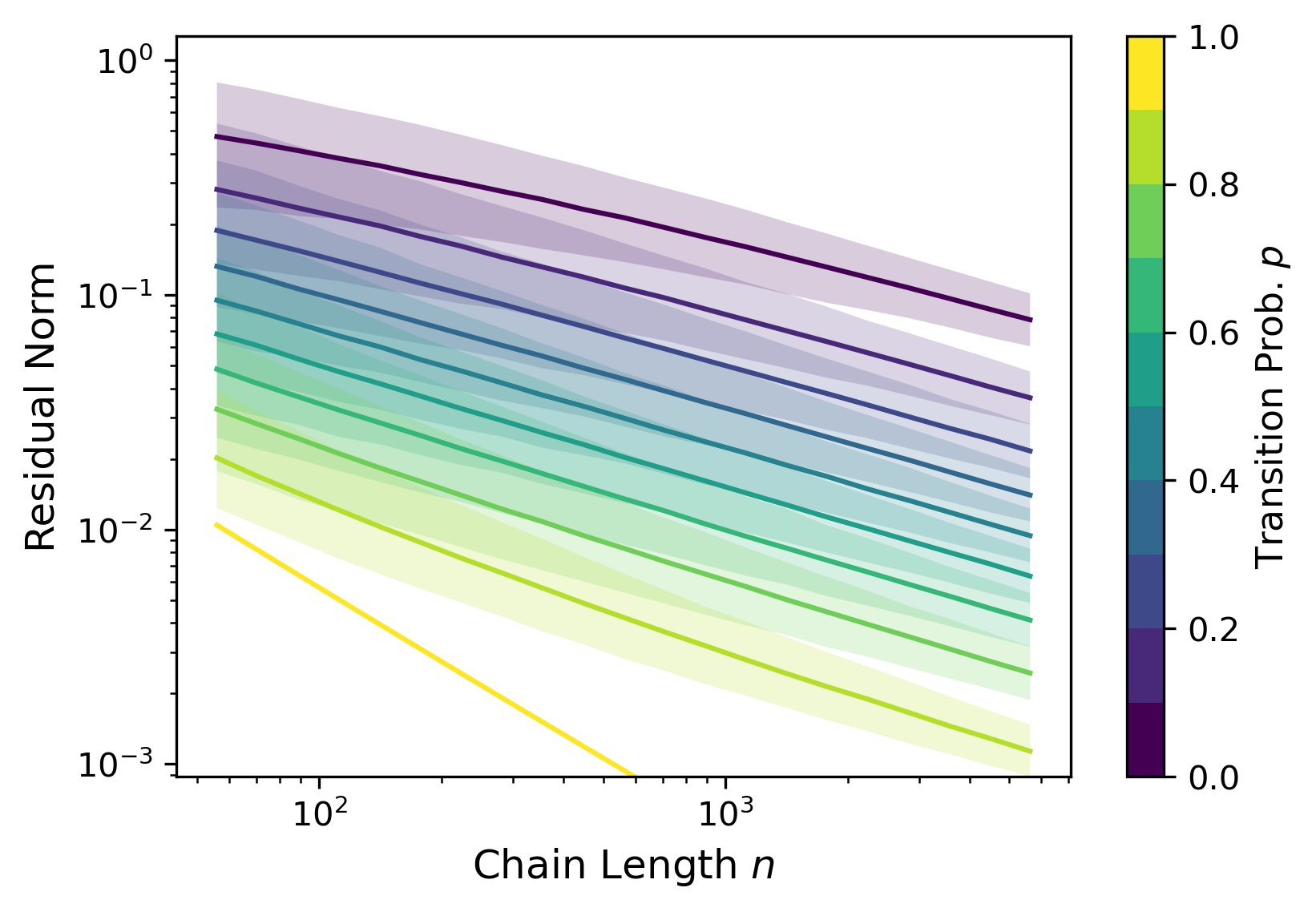}
        \caption{Residual norm $\frac{1}{n}\norm{\Gamma_\tau \bm\gf}_2^2$ 
        for $\tau_n = \sqrt{\lambda_n}$.}
        \label{fig:markov:norms}
    \end{subfigure}
     \caption{Experiment results for the 2-state, $p$-probability Markov chain. 10000 chains were simulated for each $p\in\{k/10\}_{k=1}^{10}$. Shown in subfigures are median results with 95\% probability intervals shaded in the corresponding colors. In the case of $p=1$, there is no shading.}
      \label{fig:marokv}
\end{figure}

Next we consider the case where $\Xc$ is compact and $\mathcal{K}$ is continuous, that is, kernel $\Kc$ is a Mercer kernel. Then, under the assumption that $\{x_i\}$ are i.i.d. draws, the sampling operator associated with $K$ converges compactly,
%For $x_i$ i.i.d. and $\mathcal{K}$ uniformly continuous, the sampling operator associated with $K$ converges compactly, 
almost surely, to an integral operator $T_{\Kc}:L^2(\mathcal{X})\to L^2(\mathcal{X})$~\cite[Proposition 11-13]{Luxburg08}. This in turn implies that as long as $r_n \to \infty$, we will have $\lambda_{r_n} \to 0$.
% \begin{prop}\label{prop:eig_dec}
%     Let $\Xc$ be compact and $\mathcal{K}\in L^2(\Xc\times\Xc)$. Then for i.i.d. $x_i$ and $k_n\to\infty$,
%     \[\lim_{n\to\infty} \lambda^{(n)}_{k_n} = 0\quad \text{almost surely},\]
%     where $\lambda^{(n)}_k$ is the $k$th eigenvalue of the $n\times n$-kernel matrix $K^{(n)}_{i,j} = \mathcal{K}(x_i,x_j)/n$.
% \end{prop}
%A final corollary tying back to our original Proposition~\ref{prop:km_err}.
%Combining Propositions~\ref{prop:km_err} and~\ref{lem:beta_lem},
Combined with Proposition~\ref{lem:beta_lem}, this lead to the following consistency result for the one-step procedure:
\begin{cor}\label{corr:asym_km_err}
    Consider a Mercer kernel and i.i.d. sample $\{x_i\}$. Let the regularization parameter $\tau = \tau_n$ be chosen such that the first and third term in~\eqref{eq:decomp:1} go to zero and $\xi_n = o(1)$. Further suppose that $\liminf r_n / n > 0$ and $\liminf \beta_n > 0$. Then,
    \[
    \lim_{n\to\infty}\mathbb{E}_{\bm \varepsilon}\, {\rm MSE}(\bm f^*, \bm \fh) = 0. %\leq \sigma^2/\gamma^2.
    \]
\end{cor}

%\aaa{Potentially revise below later ...}
%That is to say, the misclassification error after one step of the \altmin algorithm asymptotically approaches what is expected of $k$-means when applied to noisy centers $\{\mu^*_{z_i^*} + \varepsilon_i\}$. As such, one can show clustering consistency for the \altmin algorithm using the usual $k$-means consistency assumptions.

To provide intuition for these results, 
%For intuition on Proposition~\ref{lem:beta_lem} and Corollary~\ref{corr:asym_km_err}, 
we provide an example analyzing the decay of the spectral survival function $S_{g^*}(t)$ in a general two-class signal.
\begin{exa}\label{ex:1}
  Consider step signal $\bm g^*\in \{-1,1\}^n$ generated from an $n$-length, 2-state Markov chain with transition probability $p$. For estimation, we consider the following RKHS:
  \begin{equation}\label{eq:H1}
  \Hb^1[0,1] = \bigl\{f:[0,1]\to\reals \mid \; \text{$f$ is abs. cts., }\; 
       \norm{\partial_xf}_{L^2} < \infty, \; f(0) = 0 \bigr\}.
\end{equation}
This RKHS has kernel $\mathcal{K}(x,x') = \min(x,x')$. In this example, we assume the data is sampled at regularly spaced intervals with $x_i = i/ n$. 

The RKHS $\Hb^1[0,1]$ organizes functions by roughness through the Hilbert-norm $\norm{f}_{\Hb^1} = \norm{\partial_xf}_{L^2}$. Hence, signals $\bm g^*$ produced by chains with high transition probabilities are expected to have a larger corresponding Hilbert-norm and, intuitively, a rapidly decaying spectral survival function. 
This intuition is corroborated in Figure~\ref{fig:marokv}, where the survival function $S_{g^*}(t)$ and residual norm $\frac{1}{n}\norm{\Gamma_\tau \bm\gf}_2^2 = \frac{1}{n}\sum_{i=1}^n 
        h(\lambda_i; \tau_n) \,\gf_i^2$
are plotted for various transition probabilities. One observes that   as the transition probability increases, the tail decay of the survival function becomes sharper (Figure~\ref{fig:markov:cdfs}) and the norm decay steeper (Figure~\ref{fig:markov:norms}).

% % \[
% % S(t) =\sum_{i=1}^{n} \frac{\gf_i^2}{n}\cdot 1\{\lambda_i > t\},
% % \]
% of chains with high transition probabilities $p$ are shown to have steeper tail decays (Fig.~\ref{fig:markov:cdfs}) and faster norm decays (Fig.~\ref{fig:markov:norms}) than their low transition counterparts. %The decays in norm are in line with Proposition~\ref{lem:beta_lem}, which tells us signals with larger $\beta_n$ parameters will have quicker decays in residual norm, up to a certain exponent factor. 

The kernel matrix $K$ 
%associated with the equi-spaced case of the min-kernel
 %can be shown to have an minimum
 in this case has minimum eigenvalue $\lambda_n \approx (4n)^{-1}$.  %Moreover, with high probability, $r_n = n$ for step signals generated from the Markov chain. Then, 
 For the regularization choice $\tau_n = \sqrt{\lambda_n}$ shown in Figure~\ref{fig:markov:norms},
 %$\xi_n = \lambda_n / \tau_n = \sqrt{\lambda_n}$,
 the quickest rate of decay guaranteed by Proposition~\ref{lem:beta_lem}---namely, $\mathcal O((\lambda_n/\tau_n)^2)$---will be on the order of $\mathcal{O}(n^{-1})$. %We see that 
 This rate is attained in the log-log plot of Figure~\ref{fig:markov:norms} where the curve associated with chain transition probability $p=1$ shows a linear slope %of roughly
 of $-1$. 

%Also shown in 
Figure~\ref{fig:marokv} also provides %is 
evidence that the conditions of Corollary~\ref{corr:asym_km_err} are met for this general signal class. The survival function plots in Figure~\ref{fig:markov:cdfs} show natural tail decays for all probabilities $p$ at $n=1000$ and the stable linear decays of Figure~\ref{fig:markov:norms} show that the $\liminf$ conditions on $r_n/n$ and $\beta_n$ %do seem to be 
are attainable for a general signal model. 
\end{exa}

\begin{figure}[h]
    \begin{subfigure}{0.5\linewidth}
        \centering
        \includegraphics[width=0.95\linewidth]{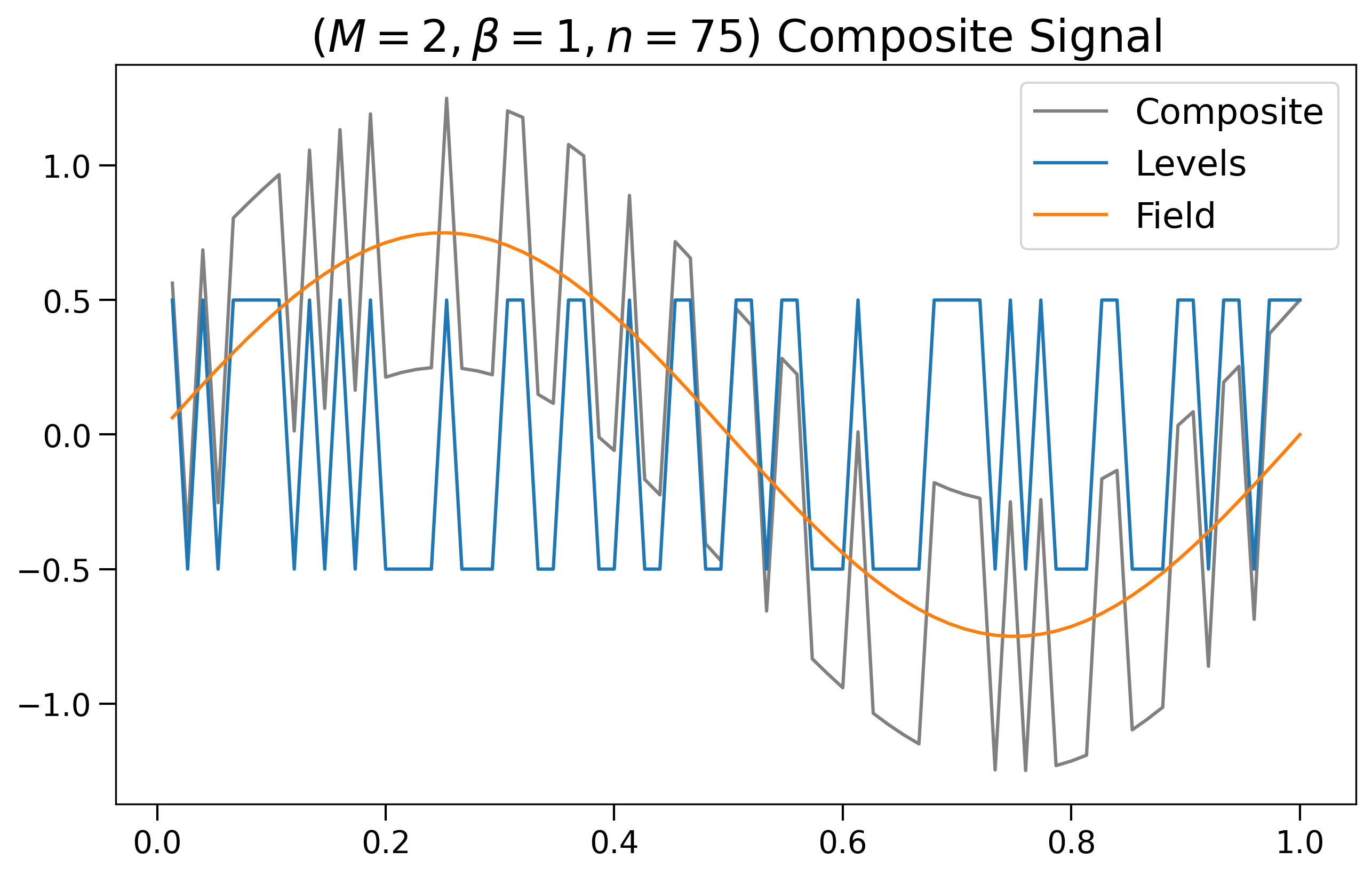}
        \caption{Smoother field}
        \label{fig:smooth}
    \end{subfigure}
    \begin{subfigure}{0.5\linewidth}
        \centering
        \includegraphics[width=0.95\linewidth]{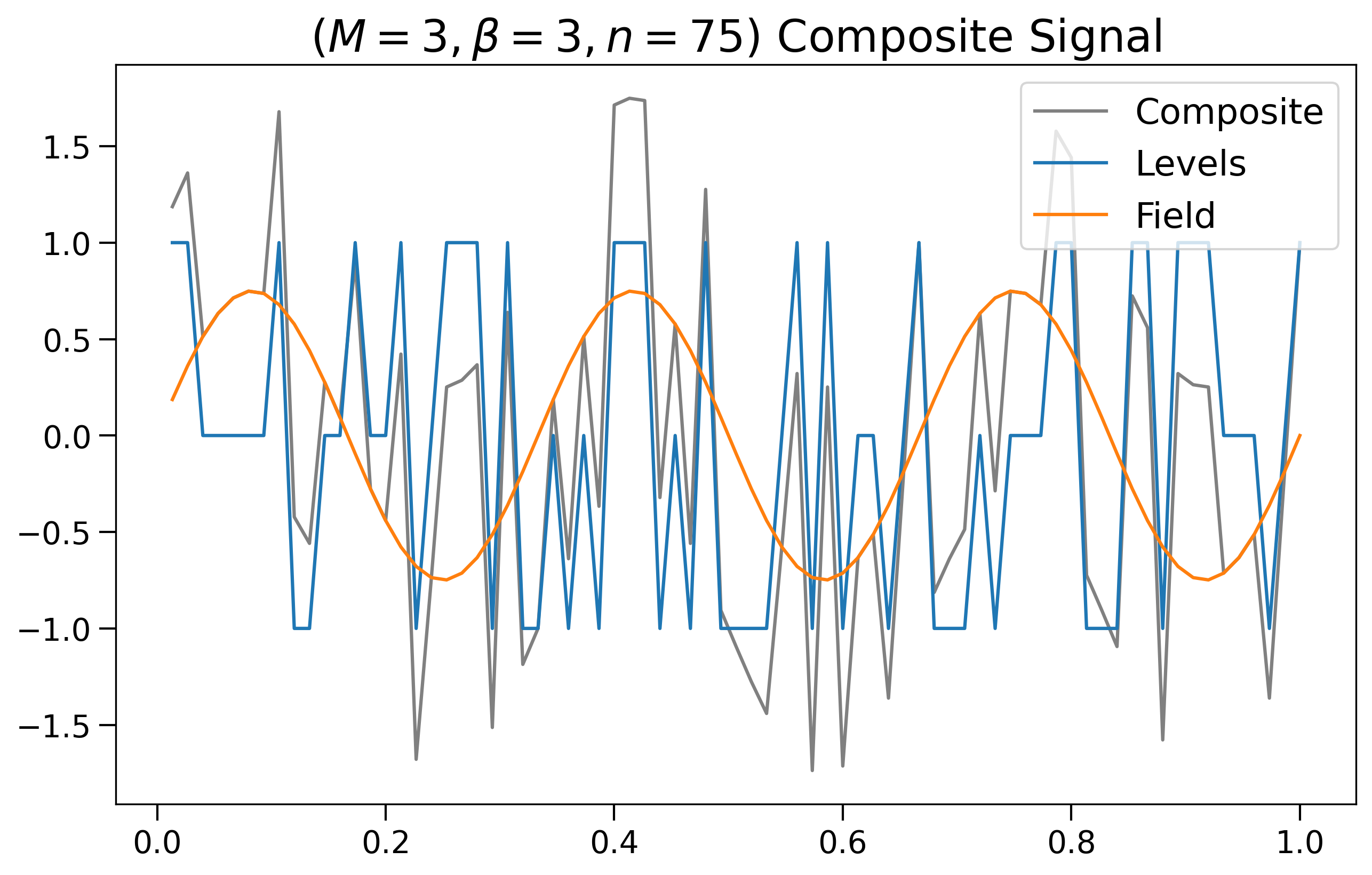}     
        \caption{Rougher field}
        \label{fig:rough}
    \end{subfigure}
    \caption{Signal, field and composite observation simulated from (\ref{eq:noiseless_rand}) for two and three classes.}
    \label{fig:simul_ex}
\end{figure}

\section{Experiments}

We now provide experimental results on the performance of the \altmin algorithm. First, we consider simulated data from 
%For simulation, we consider 
an $M$-class data generating process on $\mathcal{X} = [0,1]$ where data $X^{(n)} = \{i/n\}_{i=1}^n$ is equispaced, cluster labels $z_i^* \sim {\rm Unif}([M])$ are uniformly distributed, and the step and smooth components follow
\begin{equation}\label{eq:noiseless_rand}
    \mu_k^* = k-\frac{M+1}{2},\qquad f_\beta^* (x) = \frac{3}{4} \,\sin(2\pi\beta x).
\end{equation}
The min kernel from Example~\ref{ex:1} was chosen for estimation due to its sinusoidal eigenfunctions. Given the equispaced data, the smallest radius $\rho$ that guarantees the connectivity condition of Theorem~\ref{thm:misclass:v1} is
\begin{align*}
\min_{i\neq j} d_{\mathcal{K}}(x_i,x_j)= \sqrt{(i+1)/n - 2i/n + i/n} =  n^{-1/2},
\end{align*}
where the kernel-metric $d_{\mathcal{K}}(x,x')$ was defined in Example~\ref{ex:kernel_dist}. The Hilbert-norm of $f^*_\beta$ can be computed using inner product $\langle f, g\rangle_{\Hb^1} = \int_0^1 \partial_x f(x)\,\partial_x g(x)\,dx$. Evaluating this norm gives the following worst-case bound on the modulus of continuity of $f^*_\beta$,
\begin{equation}\label{eq:RKHS_omega}
    \omega(\rho_{\rm min}) \le \norm{f^*_\beta}_{\Hb^1}\cdot \rho_{\rm min} \leq \frac{3\sqrt{2}}{4}\pi \beta \cdot n^{-1/2}.
\end{equation}
Finally, a noisy recovery setting will be considered where i.i.d. noise $\eps_i \sim\mathcal{N}(0,\sigma^2)$ is added to mixed observations $f_\beta^*(x_i) + \mu_{z_i^*}^*$.

In both recovery settings, sample size is grown in roughly exponential manner starting from $n=25$ to $n = 3600$. At each sample size $n$, a total of 100 datasets $(X^{(n)}, \bm y)$ were simulated. Accuracy and deviation results at each $n$ were calculated using the mean score of the 100 datasets. 

\subsection{Simulation Experiments}\label{sec:simul:experiments}

\begin{figure}[t]
    \begin{subfigure}{0.5\linewidth}
        \centering
        \includegraphics[width=0.95\linewidth]{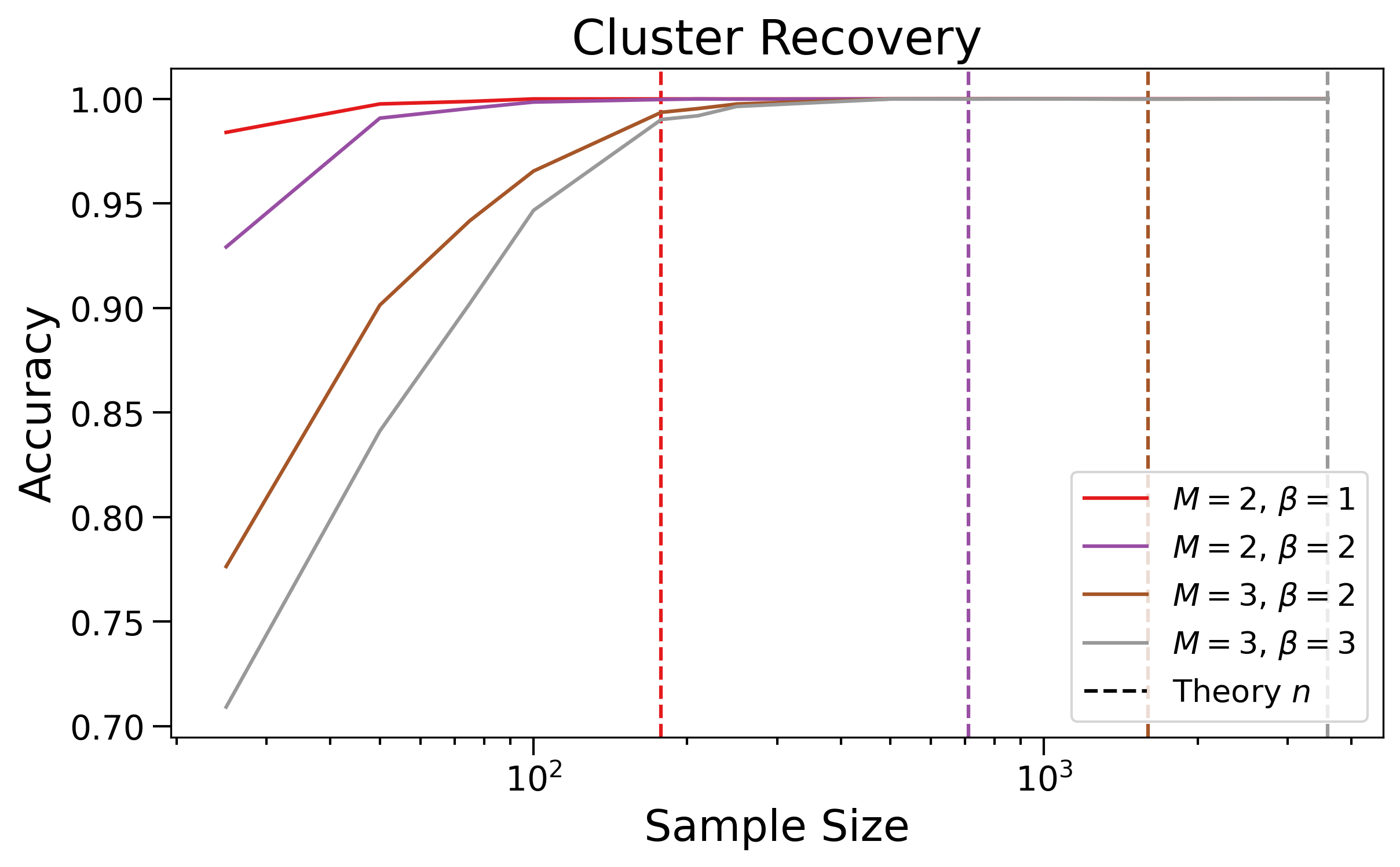}
        \caption{(Noiseless) Classification curves}
        \label{fig:noiseless:acc}
    \end{subfigure}
    \hfill
    \begin{subfigure}{0.5\linewidth}
        \centering
        \includegraphics[width=0.95\linewidth]{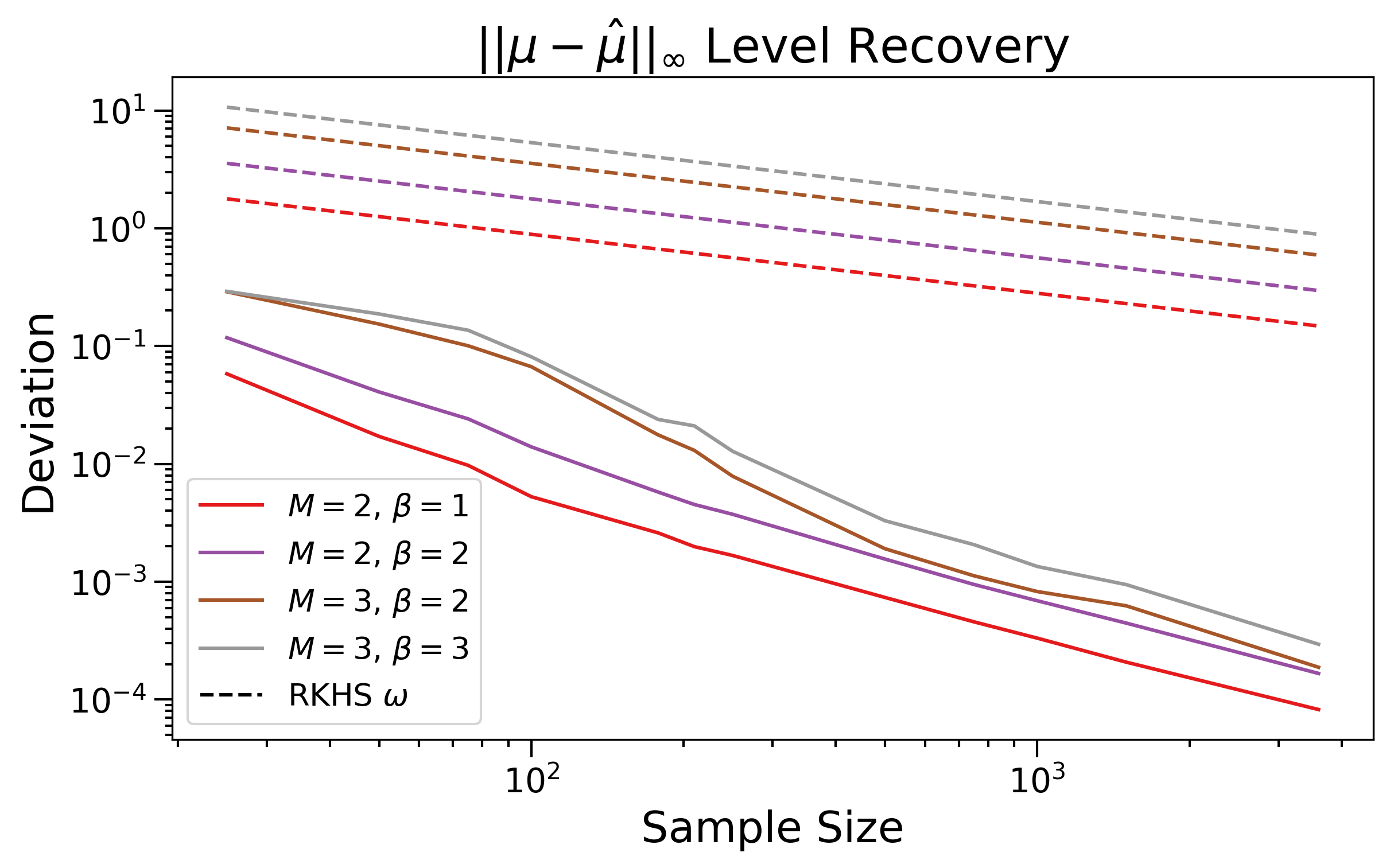}
        \caption{(Noiseless) Deviation curves}
        \label{fig:noiseless:dev}
    \end{subfigure}
     \caption{\altmin recovery results for a noiseless simulated setting. Worst-case theory bounds are shown as dashed lines for each of the different settings.}
      \label{fig:simul:noiseless}
\end{figure}

\begin{figure}
    \begin{subfigure}{0.5\linewidth}
        \centering
        \includegraphics[width=0.95\linewidth]{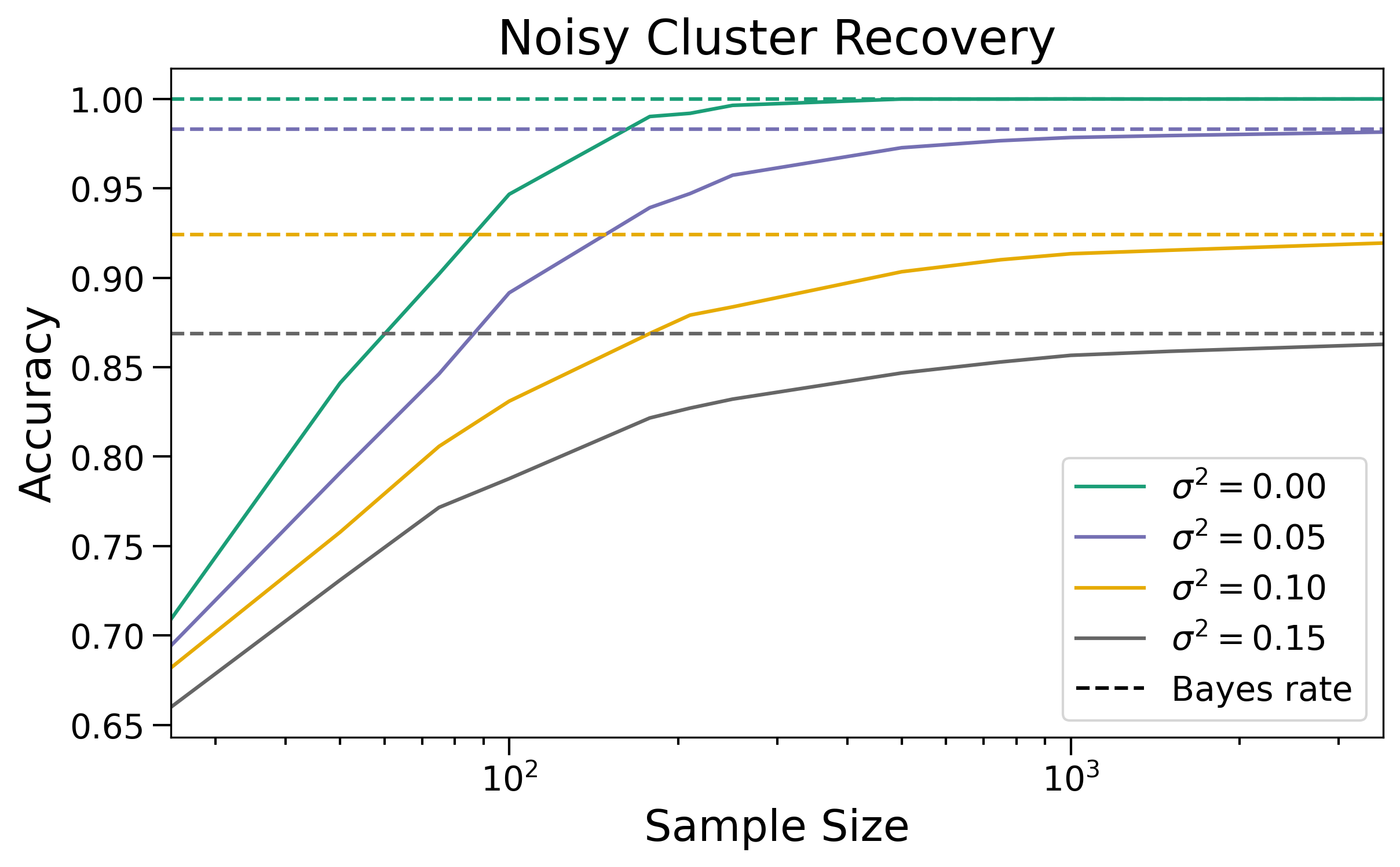}
        \caption{(Noisy) Classification curves}
        \label{fig:noisy:acc}
    \end{subfigure}
    \hfill
    \begin{subfigure}{0.5\linewidth}
        \centering
        \includegraphics[width=0.95\linewidth]{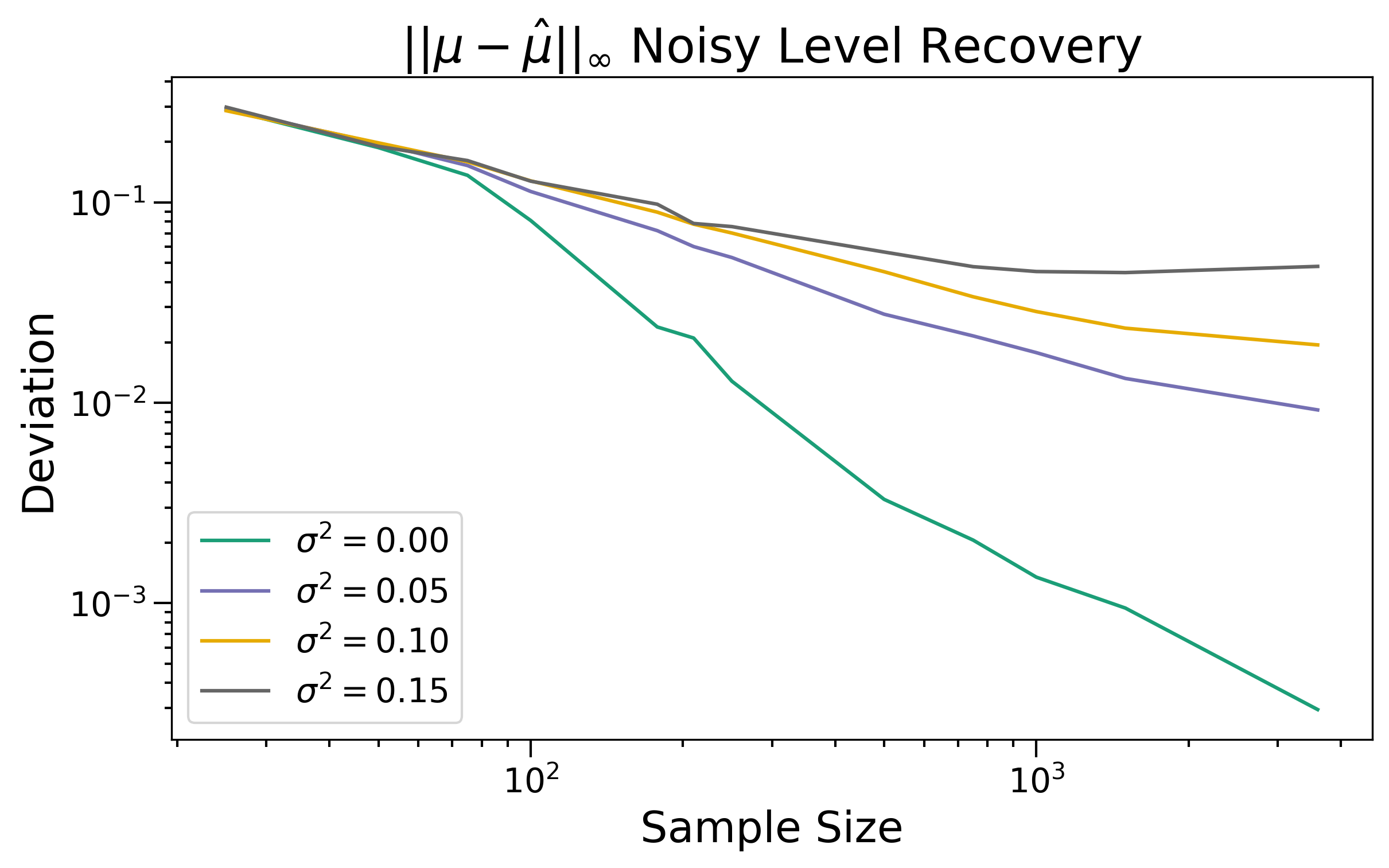}
        \caption{(Noisy) Deviation curves}
        \label{fig:noisy:dev}
    \end{subfigure}
     \caption{\altmin recovery results for a noisy simulated setting. Bayes error rates for classification are shown as dashed lines for the various noise levels.}
      \label{fig:simul:noisy}
\end{figure}

Four settings were considered for noiseless recovery: $(M,\beta) \in \{(2,1), (2,2), (3,2), (3,3)\}$. Cluster recovery and deviation results for the four settings can be found in Figure~\ref{fig:simul:noiseless}. For each setting of the optimization problem~\eqref{opt:recovery}, worst-case recovery bounds, shown dashed in Figure~\ref{fig:simul:noiseless}, were calculated using Theorem~\ref{thm:misclass:v1} and Proposition~\ref{prop:dev_bnd}. The \altmin algorithm stays well within these worst-case bounds, demonstrating the effectiveness of the simple blockwise updates for specific problem settings. 

For noisy recovery, the setting with $M=\beta = 3$ was considered at noise levels $\sigma^2 \in \{0,\,0.05,\,0.1,\,0.15\}$. Cluster recovery and deviation results for these four settings can be found in Figure~\ref{fig:simul:noisy}. In each of the noisy settings, \altmin approaches the Bayes error of what is expected for a perfect classifier.

We note that the rate at which \altmin approaches Bayes error seems faster for the cases where $\sigma^2$ is low. This may suggest that the \altmin algorithm is well-suited for smooth field, cluster recovery problems which experience low amounts of background noise.

\subsection{MRI Decontamination}\label{sec:mri_expmt}

For application, we return to the motivating MRI bias field problem. This is a real-world example where the magnitude of the inhomogeneity $f^*$ and the tissue intensity $g^*$ are much larger than the scale of the background noise $\sigma^2$~\citep{Asher10}. As we have seen in Section~\ref{sec:simul:experiments}, this is a type of problem which is a good candidate for the \altmin algorithm.

To make our experiment quantitative, we consider a 4-class, strongly-biased variant of the BrainWeb~\citep{Cocosco97} phantom. The field estimation step~\eqref{opt:rep_opt} is carried out using a Python spline routine \texttt{csaps}~\citep{Prilepin23}. This routine uses an RKHS tensor product of univariate smoothing splines to fit the multidimensional data. Relevant \texttt{csaps} smoothing parameters were selected using a post-fitting process. In practice, smoothing parameters would be selected using a validation set of data which corresponds to a specific coil cluster or MRI scanner.

\begin{figure}[t]
    \begin{subfigure}{0.3795\linewidth}
        \centering
        \includegraphics[width=\linewidth]{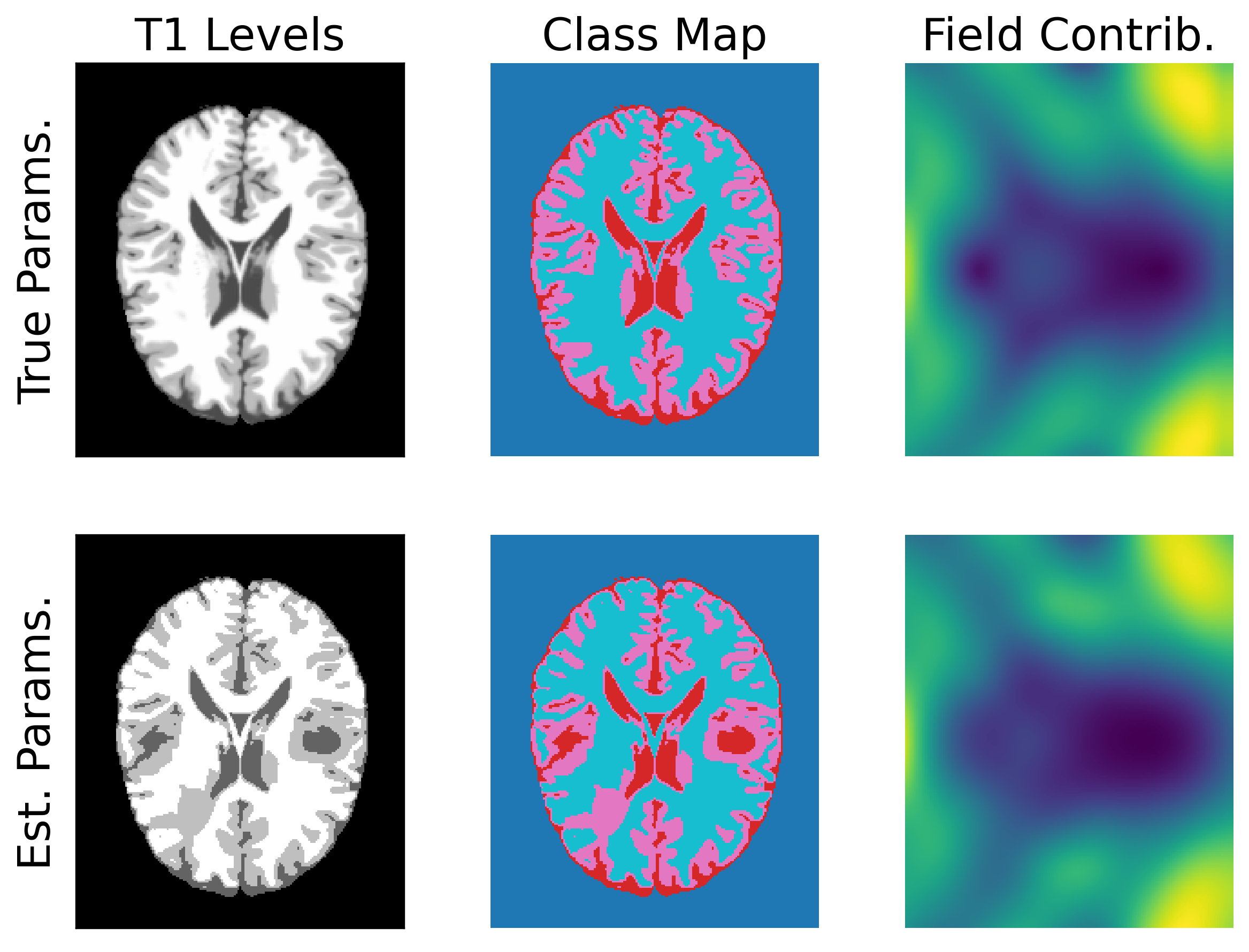}
        \caption{Single sequence recovery}
        \label{fig:t1_seq_debias}
    \end{subfigure}
    \hfill
    \begin{subfigure}{0.62\linewidth}
        \centering
        \includegraphics[width=\linewidth]{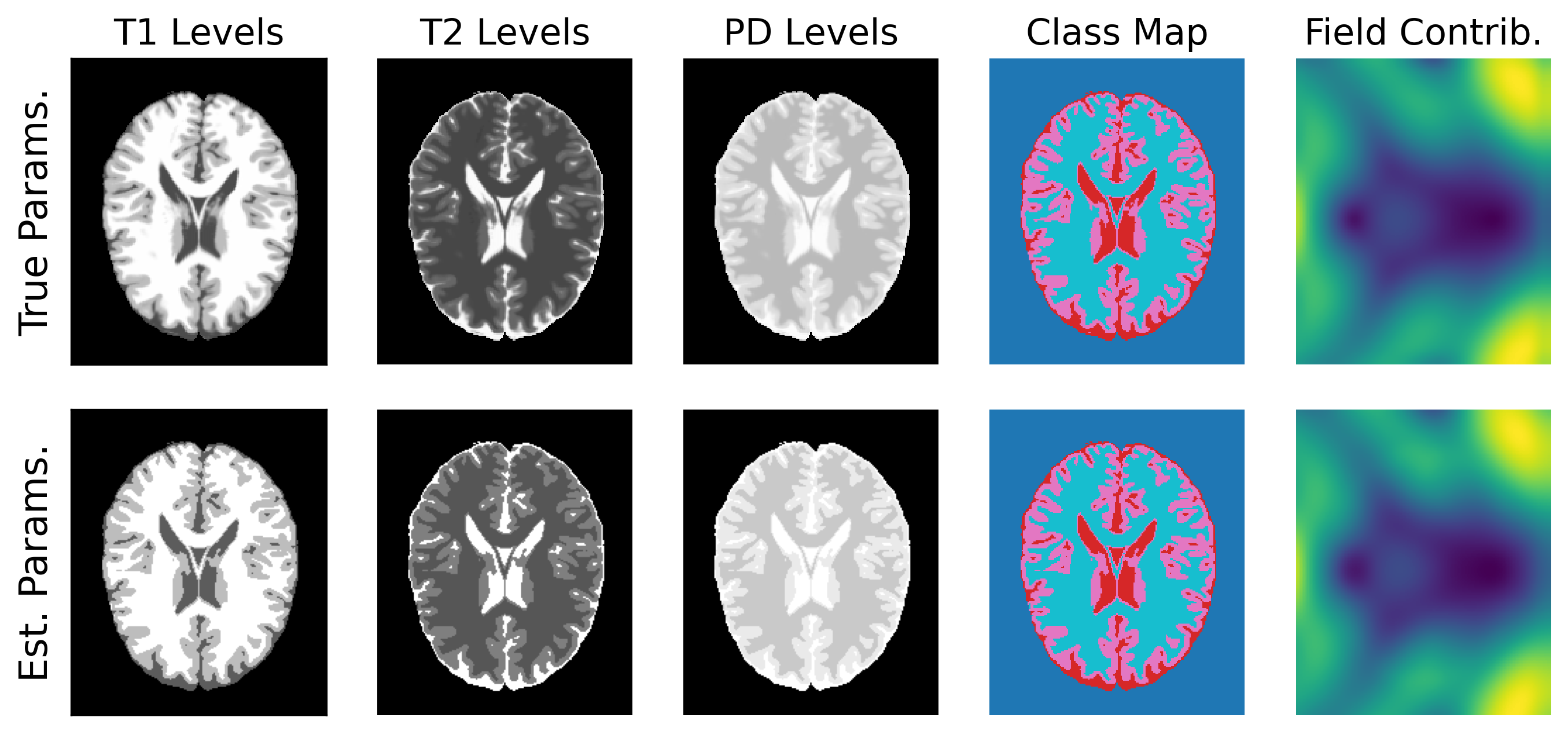}
    \caption{Multi-sequence recovery}
        \label{fig:all_seq_debias}
    \end{subfigure}
     \caption{\altmin decomposition %visualizations 
     for the biased BrainWeb dataset. Class maps of the single sequence setting show anomalous tissue patches in areas where the field changes most rapidly.}
      \label{fig:debias}
\end{figure}

\begin{figure}
  \begin{subfigure}{0.5\linewidth}
        \centering
        \includegraphics[width=0.95\linewidth]{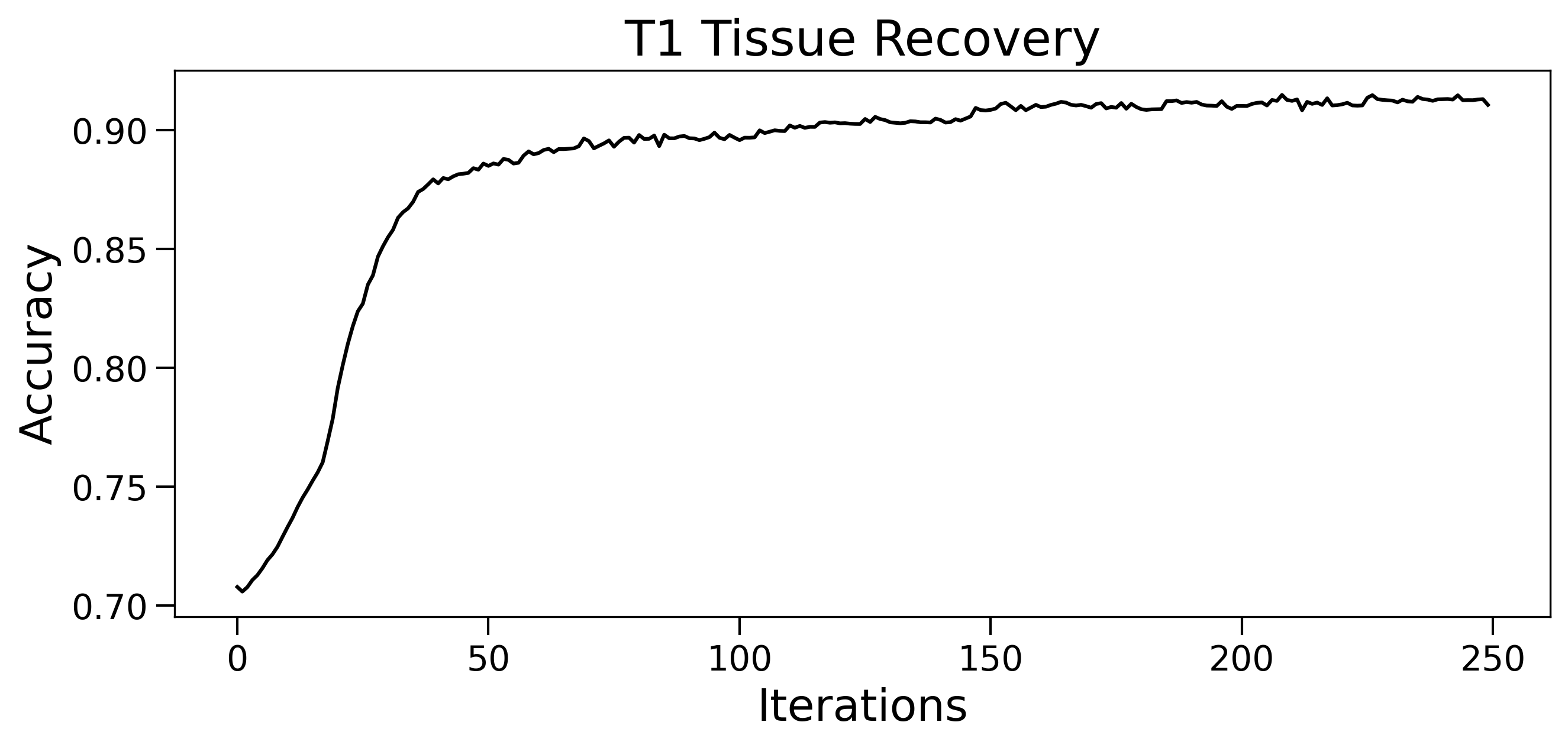}
        \caption{(Single) Accuracy vs. iteration}
         \label{fig:t1_acc_err}
    \end{subfigure}
    \hfill
    \begin{subfigure}{0.5\linewidth}
        \centering
         \includegraphics[width=0.95\linewidth]{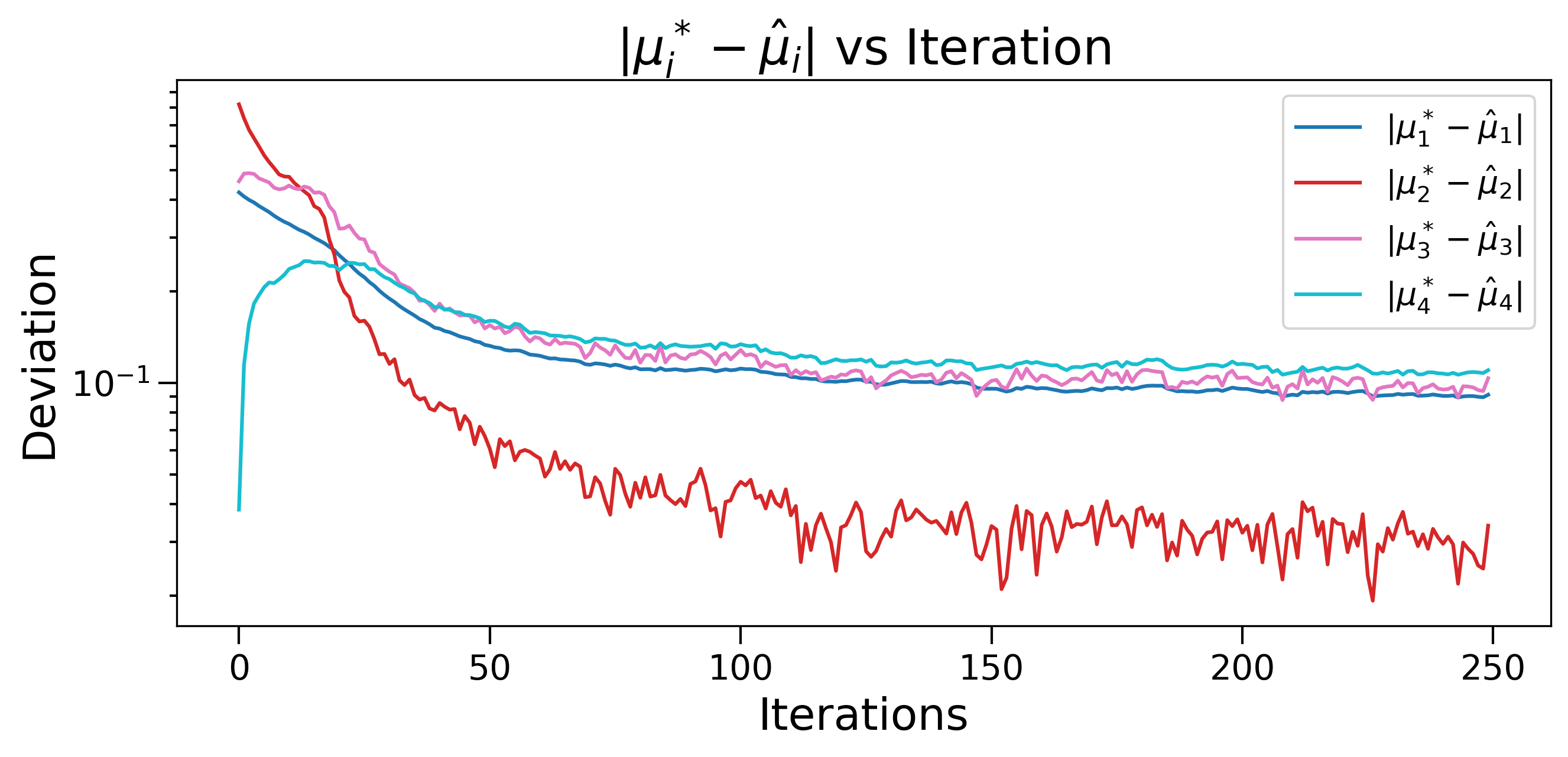}
        \caption{(Single) Deviation vs. iteration}
        \label{fig:t1_lvl_err}
    \end{subfigure}
    \begin{subfigure}{0.5\linewidth}
        \centering
         \includegraphics[width=0.95\linewidth]{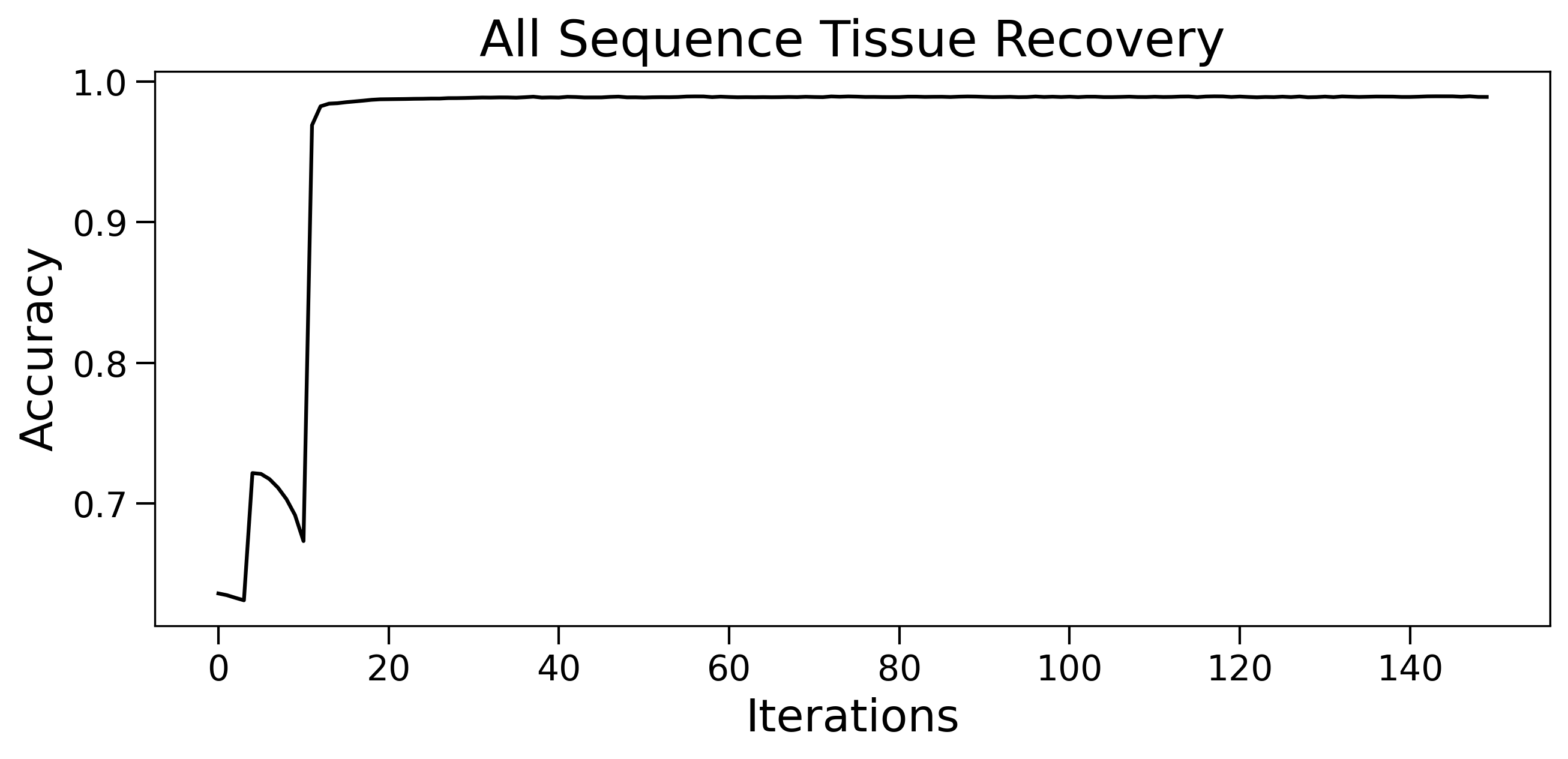}
        \caption{(Multi) Accuracy vs. iteration}
         \label{fig:all_acc_err}
    \end{subfigure}
    \hfill
    \begin{subfigure}{0.5\linewidth}
        \centering
        \includegraphics[width=0.95\linewidth]{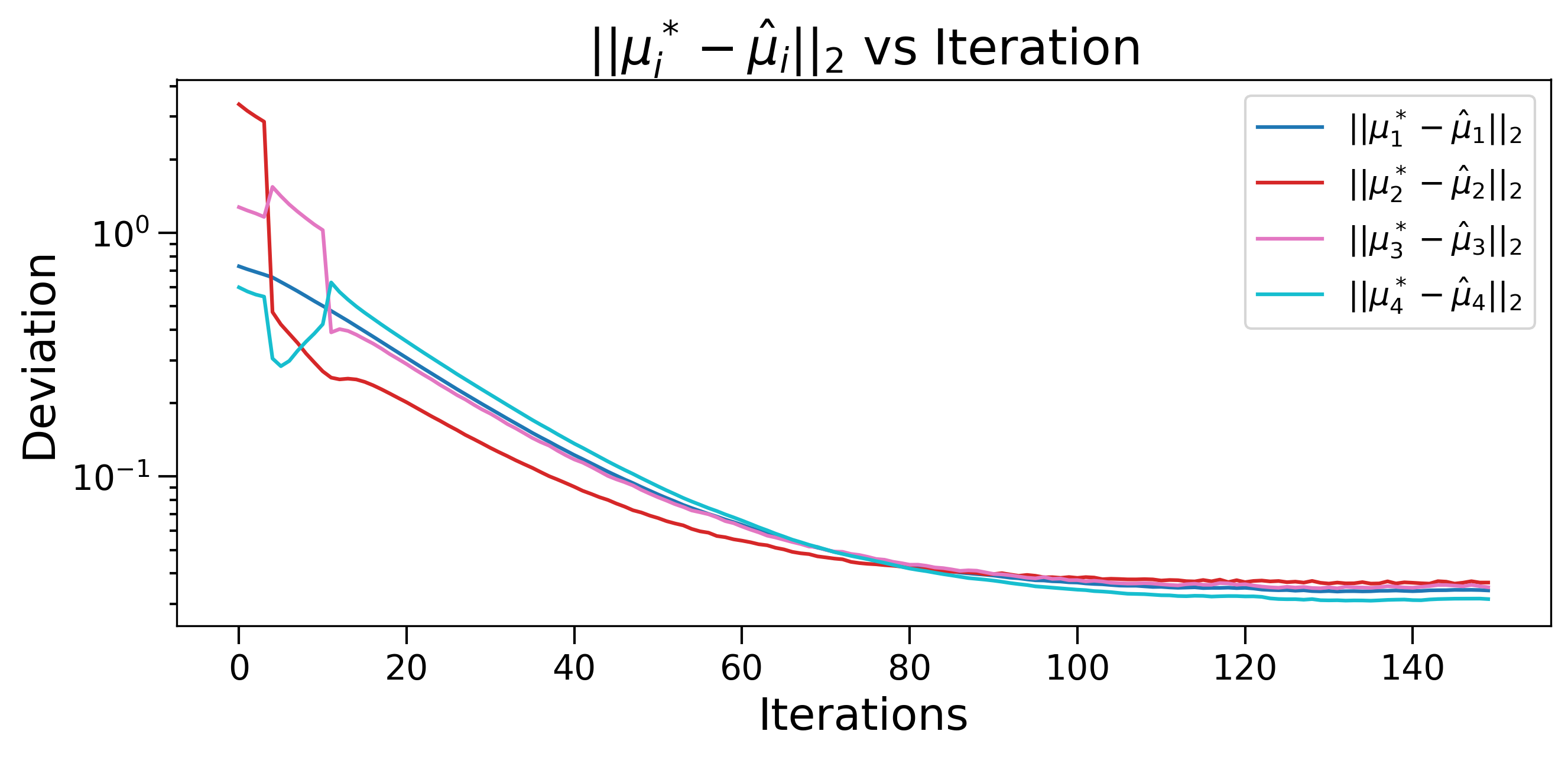}
        \caption{(Multi) Deviation vs. iteration}
        \label{fig:all_lvl_err}
    \end{subfigure}
    \caption{Cluster and level accuracy of the \altmin %optimization results for 
    algorithm on the biased BrainWeb phantom. Final accuracies for single and multi-sequence settings are 91.07\% and 98.91\% respectively. Level deviations in the multi-sequence setting are calculated with respect to the vector 2-norm.}
    \label{fig:mult:results}
\end{figure}

For implementation, we consider modeling the bias field for both single sequence and multi-sequence scans. In a multi-sequence scan, it is understood that the bias field does not vary much between sequences~\citep{Belaroussi06}. For this reason, we consider the following general $p$-sequence data model
\[
\bm y(x) = f^*(x)\cdot \bm \mu^*(x), \quad\text{for}\;x\in\mathcal{X}
\]
where levels $\bm\mu^*(x)$ take value in $\mathbb{R}^p$ and bias $f^*(x)$ is still a scalar function.

Bias field and tissue decomposition results for the single and multi-sequence setting can be found in Figure~\ref{fig:debias} with the respective \altmin optimization results found in Figure~\ref{fig:mult:results}. The presence of redundant sequencing data, albeit at different intensity scalings, seems to significantly improve \altmin convergence as shown in Figures~\ref{fig:all_acc_err}-\ref{fig:all_lvl_err}. This also translate to %a benefit in 
an improved performance, as many of the anomalous tissue patches seen in Figure~\ref{fig:t1_seq_debias} no longer occur in Figure~\ref{fig:all_seq_debias}. %Further exploration on the affects of redundant data on \altmin recovery is left to future works.
Additional experiments comparing \altmin to other medical debiasing methods can be found in Appendix~\ref{app:add:exper}.

\section{CONCLUSION}

% In this paper, we formalized the composite signal decomposition problem for continuous contaminants and step-wise signals. We derived recovery conditions, which under this model, make use of the global topology of the data such as: connectivity, minimum true level deviation, and the degree of oscillation of the contaminant. The consideration of these quantities were natural and their interplay in recovery was intuitive, such that a high-level understanding could be straightforwardly obtained from our theoretical findings.

In this paper, we defined the problem of composite signal decomposition for continuous contaminants and step-wise signals. We outlined recovery conditions that leverage the local and global topology of the data including: connectivity, minimum true level deviation, and the degree of oscillation of the contaminant. These quantities are natural, and their roles in recovery intuitively clear, allowing for a high-level understanding to be easily derived from our theoretical finding.

% In addition to identifiability, a practical algorithm \altmin was developed for contaminants which lie in a RKHS Hilbert-norm ball. This algorithm can be seen as an extension to both kernel ridge regression (KRR) and $k$-means, with updates to each being carried out in an alternating fashion. MSE bounds for the algorithm where given in terms of spectral quantities of the data, leading to a ``one-step" consistency result in the large sample limit. Lastly, algorithm performance of \altmin was evaluated empirically on both simulated and real-world data.
Besides identifiability, we developed a practical algorithm \altmin for handling contaminants that reside within an RKHS. 
%Hilbert-norm ball. 
This algorithm can be viewed as an extension of both kernel ridge regression (KRR) and $k$-means, with updates to each being performed alternately. MSE bounds for the algorithm were provided in terms of the spectral properties of the data, leading to a ``one-step'' consistency result in the large sample limit. 

We evaluated \altmin empirically on both simulated and real-world data. In the case of simulated data, \altmin operated well within the worst-case theory bounds outlined in Section~\ref{sec:recovery}. When the data was further corrupted by noise, \altmin approached the best possible classification rates for the given data generating process.
In the real-world study, we conducted an MRI tissue recovery experiment, illustrating how tensor products of smoothing splines can be employed to estimate contaminant MRI bias fields.
Given redundant data on the same bias field, \altmin significantly enhanced clustering performance and overall optimization stability.

% We empirically evaluated \altmin on both simulated and real-world data.
% For simulated data, \altmin was shown to operate well within worst-case theory bounds obtained in Section~\ref{sec:recovery}. For data that was additionally corrupted by noise, \altmin approached best possible classifcation rates for the specific data generating process.
% For the real-world study, an MRI tissue recovery experiment was conducted showing how tensor products of smoothing splines can be used to estimate contaminant MRI bias fields.  
% When provided with redundant data on the same bias field, \altmin was shown to significantly improve clustering performance and overall optimization stability.

These empirical studies, alongside the identifiability theory of Section~\ref{sec:recovery}, suggest that  step-and-smooth decompositions are attainable within worst-case optimality guarantees. Regarding application, the alternating optimization of \altmin appears well-suited for data-dense tasks, especially when data is spatially uniform and low in noise. In this context, decomposition problems akin to MRI multi-sequence recovery could be promising avenues for further applications of the \altmin algorithm.

% These empirical studies, when considered in combination with the identifiability theory of Section~\ref{sec:recovery}, suggest that %recovery of 
% achieving step-and-smooth decompositions %problems 
% is feasible within worst-case optimality guarantees. In terms of application, the alternating optimization of \altmin seems to be well-suited for data-dense tasks where data is spatially uniform and low in noise. In this light, decomposition problems which are similar %in formulation
% to MRI multi-sequence recovery may be good candidates for additional applications of the \altmin algorithm.

\bibliography{arxiv}

%\acks{All acknowledgements go at the end of the paper before appendices and references.
%Moreover, you are required to declare funding (financial activities supporting the
%submitted work) and competing interests (related financial activities outside the submitted work).
%More information about this disclosure can be found on the JMLR website.}

% Manual newpage inserted to improve layout of sample file - not
% needed in general before appendices/bibliography.

%\newpage
%\clearpage
\appendix
\section{Identifiability Proofs}
\label{app:ident:proofs}
Any optimal candidate solution $(\widehat f, \widehat \mu, \widehat z)$ to~\eqref{opt:recovery} which is fit to data $\{y_i\}$ generated from~\eqref{eq:noiseless} must satisfy
\begin{align}\label{eq:opt:equality}
    f^*(x_i) + \mu^*_{z_i^*} = \fh(x_i) + \muh_{\zh_i},\quad\text{for all}\;i \in [n]. %= 1,\ldots,n.
\end{align}
%As
Since $\fh - f^* \in \Fc_{2\omega}(\mathcal{X})$, we may instead %consider to 
analyze the discrepancy %condition
\[
g(x_i) = \mu^*_{z_i^*} - \muh_{\zh_i},\quad\text{for all}\;i=1,\ldots,n,
\]
for $g \in \Fc_{2\omega}(\mathcal{X})$. In addition, we will assume that function $g:\mathcal{X}\rightarrow\mathcal{Y}$ takes values on a normed vector space $(\mathcal{Y},\, \|\cdot\|)$. As a result, the modulus of continuity $\omega$ will be related to the induced norm-metric as $\|g(x) - g(x')\| \leq \omega\big(d(x,x')\big)$.

\medskip
The following result is the main ingredient in the proof of Theorem~\ref{thm:misclass:v1}:
\begin{thm}\label{thm:misclass}
Suppose for $g \in \Fc_{2\omega}(\mathcal{X})$ we have $g(x_i) = \mu^*_{z_i^*} - \muh_{\zh_i}$ for all $i\in [n]$ where $\bm z^* = (z_i^*)$ and $\bm \zh = (\zh_i)$ both belong to $[M]^n$. Assume the following holds:
\begin{enumerate}[(a)]
\item $\|\mu^*_k - \mu^*_\ell\|\geq \gamma$ for all $k\neq \ell$.
\item $G_\rho(X)$ is connected for some %$\rho > 0$
$\rho$ with %which satisfies 
$2\, \omega(\rho) < \gamma / M$.
\end{enumerate}
Then for all $i,j \in [n]$ we have
\begin{equation}\label{cond:equality}
\zh_i = \zh_j \implies z_i^* = z_j^*.    
\end{equation}
\end{thm}

\begin{proof}
Start by considering the induction hypothesis that, for any path $\mathcal{P}\subseteq G_\rho(X)$ of length $T$, all element pairs $i,j \in \mathcal{P}$ satisfy (\ref{cond:equality}). The base case of $T=0$ holds trivially with $i=j$.

\begin{figure*}[t]
    \centering
    %using analogous quantities.\documentclass[crop,tikz,dvipsnames]{standalone}

\usetikzlibrary{shapes}

\colorlet{CLS1}{Apricot!60}
\colorlet{CLS2}{Cerulean!60}
\colorlet{CLS3}{JungleGreen!60}
\colorlet{CLS4}{WildStrawberry!60}
\colorlet{sel_path}{blue!70}

\providecommand\zh{\widehat z}
\tikzset{site/.style={shape=circle,draw=black, minimum size=1.25cm, thick, opacity=1}}
\resizebox{\textwidth}{!}{%
\begin{tikzpicture}[scale=0.5]
    \node[site, fill=CLS1] (A) at (0,0) {{\Large $i_{1}$}};
    \node[rectangle, draw=black, fill=CLS1, below, thick, opacity=1] at (0,-1.5) {$\phi(\zh_{i_1}) = 7$};
    \node[rounded rectangle, draw=black, below, thick, opacity=1] at (0,2.75) {{\large $t_1 = 1$}};
    
    \node[site, fill=CLS2] (B) at (5,0) {{\Large $i_{2}$}};
    \node[rectangle, draw=black, fill=CLS2, below, thick, opacity=1] at (5,-1.5) {$\phi(\zh_{i_2}) = 4$};
    
    \node[site, fill=CLS3] (C) at (10,0) {{\Large $i_{3}$}};
    \node[rectangle, draw=black, fill=CLS3, below, thick, opacity=1] at (10,-1.5) {$\phi(\zh_{i_3}) = 3$};
    
    \node[site, fill=CLS2] (D) at (15,0) {{\Large $i_{4}$}};
    \node[rectangle, draw=black, fill=CLS2, below, thick, opacity=1] at (15,-1.5) {$\phi(\zh_{i_4}) = 4$};
    \node[rounded rectangle, draw=black, below, thick, opacity=1] at (15,2.75) {{\large $t_2 = 4$}};
    
    \node[site, fill=CLS4] (E) at (20,0) {{\Large $i_{5}$}};
    \node[rectangle, draw=black, fill=CLS4, below, thick, opacity=1] at (20,-1.5) {$\phi(\zh_{i_5}) = 6$};
    
    \node[site, fill=CLS4] (F) at (25,0) {{\Large $i_{6}$}};
    \node[rectangle, draw=black, fill=CLS4, below, thick, opacity=1] at (25,-1.5) {$\phi(\zh_{i_6}) =6$};
    \node[rounded rectangle, draw=black, below, thick, opacity=1] at (25,2.75) {{\large $t_3 = 6$}};
    
    \node[site, fill=CLS1] (G) at (30,0) {{\Large $i_{7}$}};
    \node[rectangle, draw=black, fill=CLS1, below, thick, opacity=1] at (30,-1.5) {$\phi(\zh_{i_7}) = 7$};

    \path[line width=0.45mm] (A) edge node {} (B);
    \path[line width=0.45mm] (B) edge node {} (C);
    \path[line width=0.45mm] (C) edge node {} (D);
    \path[line width=0.45mm] (D) edge node {} (E);
    \path[line width=0.45mm] (E) edge node {} (F);
    \path[line width=0.45mm] (F) edge node {} (G);
    
    \path[line width=0.45mm, above] (A.10) edge[sel_path] node {} (B.170);
    \path[line width=0.45mm, above] (D.10) edge[sel_path] node {} (E.170);
    \path[line width=0.45mm, above] (F.10) edge[sel_path] node {} (G.170);
    
    \path[line width=0.45mm] (B) edge [bend left=40, sel_path, densely dashed] (D);
    \path[line width=0.45mm] (E) edge [bend left=50, sel_path, densely dashed] (F);
\end{tikzpicture}
}
    \caption{Example use of the $\phi(r)$ function and $\{(u_q,v_q)\}_{q=1}^Q$ sequence, shown in blue above, for a 4-class path of length~7. The colors denote the estimated cluster labels. This example has $\{(u_q,v_q)\}_{q=1}^Q=\{(i_1,i_2),(i_4,i_5),(i_6,i_7)\}$ with $Q = 3$.}
    \label{fig:path:construction} 
\end{figure*}
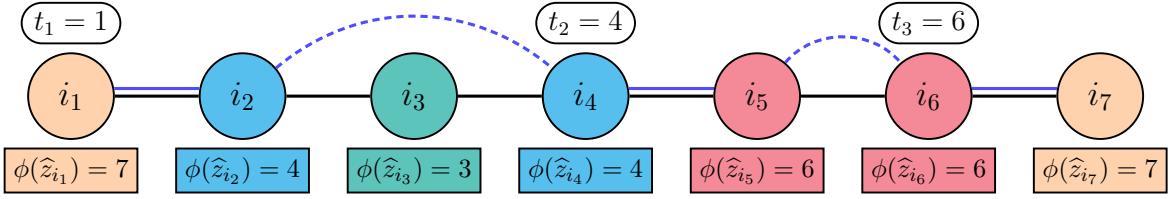

Throughout the proof, by the label of a node $i$, we mean its estimated label $\zh_i$. Consider a general path $\mathcal{P} = \{i_t\}_{t=1}^{T+1}$ of length $T+1$ inside $G_\rho(X)$. As both $\{i_t\}_{t=1}^T$ and $\{i_t\}_{t=2}^{T+1}$ are paths of length $T$, we only need to verify (\ref{cond:equality}) for $i_1$ and $i_{T+1}$. Therefore, for our induction step it is sufficient to show that $\zh_{i_1} = \zh_{i_{T+1}}$ and $z_{i_1}^*\neq z_{i_{T+1}}^*$ cannot simultaneously hold for the given assumptions (a) and (b).

For the sake of contradiction, assume $\zh_{i_1} = \zh_{i_{T+1}}$ and $z_{i_1}^* \neq z_{i_{T+1}}^*$. Under this assumption the induction hypothesis guarantees
\begin{equation}\label{cond:unique}
    \zh_{i_t}\neq \zh_{i_1}\quad\text{for}\;1<t<T+1.
\end{equation}
Note that if this was not the case with
$$\zh_{i_1} = \zh_{i_t} = \zh_{i_{T+1}}\quad\text{for some}\; 1 < t < T+1,$$
then the condition $z_{i_1}^*\neq z_{i_{T+1}}^*$ would have caused a contradiction at the earlier induction step $\max\{(T+1)-t,t-1\}$.

Next let $\mathcal{R}$ be the set of labels $\zh_{i_t}$ on path $\mathcal{P}$. Function $\phi(r)$ will be the index of the last node we see on the path from $i_1$ to $i_{T+1}$ that has label $r$, that is,
$$\phi(r) = \max_{t\in [T+1]}\{t:\;\zh_{i_t}=r\}.$$
We construct an edge sequence $\{(u_q,v_q)\}_{q=1}^Q$---where $Q$ is determined by the construction---recursively as follows: Let $(u_1,v_1) = (i_1,i_2)$ and for $q = 2,\ldots,Q,$
\begin{align*}
    (u_q,v_q) = (i_{t_q},i_{t_q + 1})\quad\text{where}\quad t_q = \phi(\zh_{v_{q-1}}).
\end{align*}
The construction continues until $t_Q = T$, so that
$(u_Q, v_Q) = (i_T, i_{T+1})$. See Figure~\ref{fig:path:construction} for a concrete example.
By construction, the labels of $v_{q-1}$ and $u_q$ are the same, while the labels of $v_{q-1}$ and $v_q$ are necessarily different. By this latter property, the labels of $v_1,\ldots,v_{Q-1}$ are distinct elements of $\mathcal{R}$. The added uniqueness condition of (\ref{cond:unique}) gives that the label of $v_Q$ is also distinct from $v_1,\ldots,v_{Q-1}$, hence $Q\leq |\mathcal{R}|$.

Using $\zh_{v_{q-1}} = \zh_{u_q}$, we obtain the decomposition
\begin{equation}\label{decomp:hat}
\muh_{\zh_{u_1}} - \muh_{\zh_{v_Q}} = \sum_{q=1}^Q (\muh_{\zh_{u_q}} - \muh_{\zh_{v_q}}).    
\end{equation}
From the induction hypothesis, $\zh_{v_{q-1}} = \zh_{u_q}$ implies $z_{v_{q-1}}^* = z_{u_q}^*$ for $2\leq q\leq Q$. This gives the decomposition
\begin{equation}\label{decom:star}
   \mu_{z^*_{u_1}}^* - \mu_{z^*_{v_Q}}^* = \sum_{q=1}^Q (\mu^*_{z^*_{u_q}} - \mu^*_{z^*_{v_q}}).   
\end{equation}
Moreover, since $u_q$ and $v_q$ are adjacent on the path, they satisfy $d(x_{u_q},x_{v_q})\leq \rho$, which by assumption (b) implies
\begin{equation}\label{eq:dev}
    \|(\mu_{z^*_{u_q}}^* - \mu_{z^*_{v_q}}^*) - (\muh_{\zh_{u_q}} - \muh_{\zh_{v_q}})\| = \|g(x_{u_q}) - g(x_{v_q})\| < \gamma / M.
\end{equation}
By assumption, $\zh_{u_1} = \zh_{v_Q}$, hence the LHS of (\ref{decomp:hat}) is zero. Then, subtracting decomposition (\ref{decomp:hat}) from (\ref{decom:star}) and using the triangle inequality, we get
$$\|\mu_{z_{u_1}^*}^* - \mu_{z_{v_Q}^*}^*\|\leq \sum_{q=1}^Q\|(\mu^*_{z^*_{u_q}} - \mu^*_{z^*_{v_q}}) - (\muh_{\zh_{u_q}} - \muh_{\zh_{v_q}})\| < Q\,\gamma / M,$$
where the second inequality is by (\ref{eq:dev}). If at the same time $z_{i_1} \neq z_{i_{T+1}}$ then $\mu^*_{z_{u_1}^*}\neq \mu^*_{z_{v_Q}^*}$, and by assumption (a), $\gamma \leq \|\mu^*_{z_{u_1}^*} - \mu^*_{z_{v_Q}^*}\|$. Hence,
$$\gamma < Q\, \gamma/ M \leq |\mathcal{R}|\, \gamma/M.$$
Since $|\mathcal{R}|\leq M$, we arrive at a contradiction. This completes the induction step. Applying our induction claim to the connected $G_\rho(X)$ completes the proof.
\end{proof}

Theorem~\ref{thm:misclass} shows that $\zh$ is a \emph{refinement} of $z^*$. But since both $\zh$ and $z^*$ have the same number of classes ($M$), the classes of $\zh$ should, in fact, coincide with those of $z^*$. This is formalized in the following corollary:

\begin{cor}
Suppose every label in $[M]$ is attained by $\bm z^*$ on $[n]$. Then under the conditions of Theorem \ref{thm:misclass}, the misclassification rate between $\bm z^*$ and $\bm \zh$ is zero.
\end{cor}
% \begin{proof}
% Let $\cluster_k = \{i: z^*_i = k\}$ and $\widehat\cluster_\ell = \{i: \zh_i = \ell\}$. By  Theorem~\ref{thm:misclass}, each $\widehat\cluster_\ell$ is entirely contained within one of $\cluster_k$. Assume that one of $\cluster_k$ contains at least two of the estimated clusters, say $\widehat\cluster_{\ell_1}$ and $\widehat \cluster_{\ell_2}$. Then, by the pigeon-hole principle there some $\cluster_{k'}$ which contains none of the estimated clusters, a contradiction. This shows that the reverse implication in~\eqref{cond:equality} also holds, finishing the proof.
% \end{proof}

% The contrapositive implication of Theorem \ref{thm:misclass} gives
% \begin{equation}\label{cond:contra}
%     z^*_i \neq z^*_j \implies \zh_{i}\neq \zh_j\quad\text{for all}\;i,j\in[n].
% \end{equation}
% Define the set of labels confused with true label $k$ as
% $$C_k = \{\ell \in [M]:\, \exists i\in [n],\, \zh_i = \ell\;\text{and}\; z^*_i = k\}.$$
% Clearly $|C_k|\geq 1$ for all $k\in[M]$, and by (\ref{cond:contra})
% $$|C_k\cap C_{k'}| = 0\quad\text{for all}\;k\neq k'.$$
% Then the pigeon-hole principle each $|C_k| = 1$. As such, there will exist a bijective mapping $\sigma:[M]\rightarrow [M]$ from $\zh$ to $z^*$ which achieves perfect misclassification.
% \end{proof}

Let us now prove Proposition~\ref{prop:dev_bnd}. Under the assumptions of Theorem~\ref{thm:misclass:v1}, we can relabel the classes of $(\bm \zh, \bm \muh)$ so that $\bm\zh = \bm z^*$. Then, it follows from~\eqref{eq:opt:equality} that
\begin{align}\label{eq:opt:equality:v2}
    \fh(x_i) - f^*(x_i) = \mu^*_{z_i^*} - \muh_{z_i^*}\quad\text{for all}\;i \in [n].
\end{align}

% Define the zero-mean, recovery problem by
% \begin{align}\label{opt:zero_mean_recovery}
%     (\widehat f, \widehat \mu, \widehat z) = \argmin_{\substack{f \in \Fc_\omega(\mathcal{X}),\\ \mu \in \mathbb{R}^M,\;z\in [M]^n}}&\; \sum_{i=1}^n(y_i - \mu_{z_i} - f(x_i))^2\\
%     &\text{subject to}\quad \sum_{i=1}^n f(x_i) = 0.\nonumber
% \end{align}

% \begin{lem}
% Suppose candidate $\fh \in \Fc_\omega(\mathcal{X})$ is zero-mean on $X = \{x_i\}_{i=1}^n$ and satisfies~\eqref{eq:opt:equality:v2}.
% %$$\fh(x_i) - f^*(x_i) = \mu^*_{z_i^*} - \muh_{z_i^*}\quad\text{for all}\;i=1,\ldots,n.$$
% Then,
% \[\max_{k\in [M]} |\mu^*_k - \muh_k| \leq 2(M-1)\cdot \omega(\lbd) + \frac{1}{n}\big|\sum_{i=1}^n f^*(x_i)\big|.\]
% \end{lem}

\subsection{Proof of Proposition~\ref{prop:dev_bnd}}

%For every $(\cluster_k, \cluster_\ell)\in G_\delta(\cluster)$ there exists $i,j\in [n]$ such that $z_i^* = k$, $z_j^* = \ell$ with $d(x_i,x_j)\leq \delta$. 
For $\delta  \geq \lbd$, the neighbor graph $G_\delta(\cluster)$ %of $\cluster$ 
%Let $G$ be minimizing graph in the definition of $\lbd$. For every $(k,\ell) \in G$ there exists $i,j\in [n]$ such that $z_i^* = k$, $z_j^* = \ell$ with $d(x_i,x_j)\leq \lbd$.  As $G$
is connected such that every $k,\ell \in [M]$ has a series of edges $\{(x_{i_t}, x_{j_t})\}_{t=1}^T$ with $d(x_{i_t},x_{j_t})\leq \delta$ such that $z_{i_1}^* = k$, $z_{j_T}^* = \ell$ and
\[
%z_{i_1}^* = k,  \quad z_{j_T}^* = \ell \quad, %\text{and}\quad
%(z_{i_1}^*, z_{j_T}^*) = (k,\ell)\quad, %\text{and}\quad
z_{j_{t-1}}^* = z_{i_t}^* \neq  z^*_{j_t}. %\quad\text{for}\;2\leq t \leq T,
\]
%\aed{for all the intermediate points on the path.}
%for $2 \le t \le T$.
In particular, the condition $ z_{i_t}^* \neq  z^*_{j_t}$ ensures $T \le M-1$. Let $\delta = \lbd$ and with the shorthands $g = \fh - f^*$ and $\Delta_k = \mu_k^* - \muh_k$, we have $g(x_i) = \Delta_{z^*_i}$ for all $i \in [n]$. Then, the following inequality holds for all $k,\ell\in [M]$,
\begin{align*}
    \|\Delta_k - \Delta_\ell\|
    &\leq \sum_{t=1}^T \|g(x_{i_t}) - g(x_{j_t})\|\\
    &\leq T\cdot 2\omega(\lbd) \leq 2(M-1)\cdot \omega(\lbd).
\end{align*} 
Letting %$\pi_k = \frac{1}{n}(\sum_{i:\,z_i^* = k} 1)$ 
$\pi_k = |\cluster_k|/n$ be the proportion of class $k$, then
\begin{align*}
    \Bigl\|\Delta_k - \frac{1}{n}\sum_{i=1}^n g(x_i)\Bigr\|
    &= \Bigl\|\Delta_k - \sum_{\ell=1}^M \pi_\ell \Delta_\ell\Bigr\|\\
    &\leq 
    \sum_{\ell=1}^M \pi_\ell \|\Delta_k - \Delta_\ell\|.
\end{align*}
Since $\fh$ is assumed zero-mean, $\frac{1}{n}\|\sum_{i=1}^n g(x_i)\| = \frac{1}{n}\|\sum_{i=1}^n f^*(x_i)\|$. Putting the pieces together, using the triangle inequality and noting that $\sum_\ell \pi_\ell = 1$ finishes the proof.

\section{Supplement to Section~\ref{sec:one:step:analysis}}
\subsection{Proof of Proposition~\ref{lem:beta_lem}}
Let $\bblambda$ be the discrete random variable defined by
\begin{equation}
    \bblambda = \begin{cases}
    \lambda_i,&\text{w.p.}\;\;\gf_i^2/(n\norm{\bm g^*}_\infty^2)\\
    0,&\text{w.p.}\;\; 1 - \norm{\bm g^*}_2^2  /(n\norm{\bm g^*}_\infty^2).
    \end{cases}
\end{equation}
Further define $\psi(\lambda) = (1+ \tau/\lambda)^{-2}$, then
\begin{align*}
    \frac{1}{n}\norm{\Gamma_\tau \bm\gf}_2^2 &= \sum_{i=1}^n \frac{\gf_i^2}{n} \psi(\lambda_i)\leq \norm{\bm g^*}_\infty^2 \,\mathbb{E}[\psi(\bblambda)].
\end{align*}
Define $r$ and $\beta$ as before. Function $\psi(\cdot)$ is non-negative and monotone on $[0,\infty)$ so 
\begin{align}\label{eq:tail_int}
    \mathbb{E}[\psi(\bblambda)] &= \int_0^\infty {\rm Pr}(\psi(\bblambda) > t)\,dt\notag\\
    &= \psi(\lambda_r) + \int_{\psi(\lambda_r)}^{\psi(\lambda_1)} {\rm Pr}(\bblambda > \psi^{-1}(t)) dt \notag\\
    &\leq \psi(\lambda_r) + \int_{\psi(\lambda_r)}^{\psi(\lambda_1)} \bigg(\frac{\lambda_r}{\psi^{-1}(t)}\bigg)^\beta\,dt
\end{align}
Denote the last right-hand side integral as $\mathcal{I}(\beta)$. Integral $\mathcal{I}(\beta)$ is monotone decreasing with $\mathcal{I}(\beta) < \mathcal{I}(\beta')$ for $0\leq \beta' < \beta$. %That is to say, for $\beta$ which is known beforehand, $\mathcal{I}(\beta)$ can be upper bounded with rational $\beta' \leq \beta$. 
Next, the inverse function $\psi^{-1}$ can be lower bounded as
\begin{align}\label{eq:psi_lbd}
    \psi^{-1}(t) = \frac{\tau}{t^{-1/2}-1}\geq \tau\,t^{1/2}(1-t)^{-1/2}.
\end{align}
Restricting focus to $\beta \in [0,2)$ and applying~\eqref{eq:psi_lbd} to integral $\mathcal{I}(\beta)$ yields
\begin{align*}
    \mathcal{I}(\beta) &\leq (\lambda_r/\tau)^\beta \int_{\psi(\lambda_r)}^{\psi(\lambda_1)} t^{-\beta/2}(1-t)^{\beta/2}\,dt\\
    &\leq (\lambda_r/\tau)^\beta \int_0^1 t^{-\beta/2}(1-t)^{\beta/2}\,dt\\
    &= (\lambda_r/\tau)^\beta\,\frac{\Gamma(1-\beta/2) \, \Gamma(1+\beta/2)}{\Gamma(2)}.
\end{align*}
Identity $\Gamma(z) = \Gamma(1+z)/z$ can be used with $\beta \in [0,2)$ to get
\[
\frac{\Gamma(1-\beta/2) \, \Gamma(1+\beta/2)}{\Gamma(2)} \leq \frac{2}{2 -\beta}.
\]
Lastly since $\psi(\lambda) \leq (\lambda/\tau)^2$ and $\mathcal{I}(\beta)$ is monotone decreasing in $\beta$, we have
\begin{equation}\label{bnd:full_beta}
    \frac{1}{n} \norm{\Gamma_\tau \bm \gf}_2^2 \leq 2\norm{\bm g^*}_\infty^2\cdot\max\bigg\{(\lambda_r/\tau)^2,\,\inf_{\beta'\in [0,\beta)} \frac{2}{2-\beta'} \cdot (\lambda_r/\tau)^{\beta'}\bigg\}.
\end{equation} 

Let $\xi \coloneqq \lambda_r/\tau$. For $\xi < 1$, function $h(x) = \xi^x/(2-x)$ achieves global minimum at $x^* = 2 - (\log \frac{1}{\xi})^{-1}$. This %minima 
minimum is non-negative for $\xi < e^{-1/2}$. That is, when $\beta = 2$, we have $\beta'$ approaching 2 as $\lambda_r/\tau$ approaches 0. More specifically, we obtain the following simplification to~\eqref{bnd:full_beta}
\begin{equation}
    \frac{1}{n} \norm{\Gamma_\tau \bm \gf}_2^2  \leq 4\norm{\bm g^*}_\infty^2 \log (1/\xi)\, \xi^{2 - (\log \frac{1}{\xi})^{-1}}.
\end{equation}

%Note, in the case where $\beta = 2$, the bound~\eqref{bnd:full_beta} will have minimizing argument $\beta'$ approach $2$ as $\lambda_r/\tau$ tends to 0. For example, given $(\lambda_r/\tau) = O(n^{-\alpha})$ decaying with sample size and $2 - \beta' = \epsilon$, we have that the inequality
%\[
%\frac{1}{\epsilon} n^{-\alpha(2 - \epsilon)} < \frac{1}{2\epsilon} n^{-\alpha(2 - %2\epsilon)}
%\]
%is achieved for all $n > 2^{\alpha/\epsilon}$.

\subsection{Sobolev Kernel Rates}

\begin{figure}[t]
    \begin{subfigure}{0.49\linewidth}
        \centering
        \includegraphics[width=\linewidth]{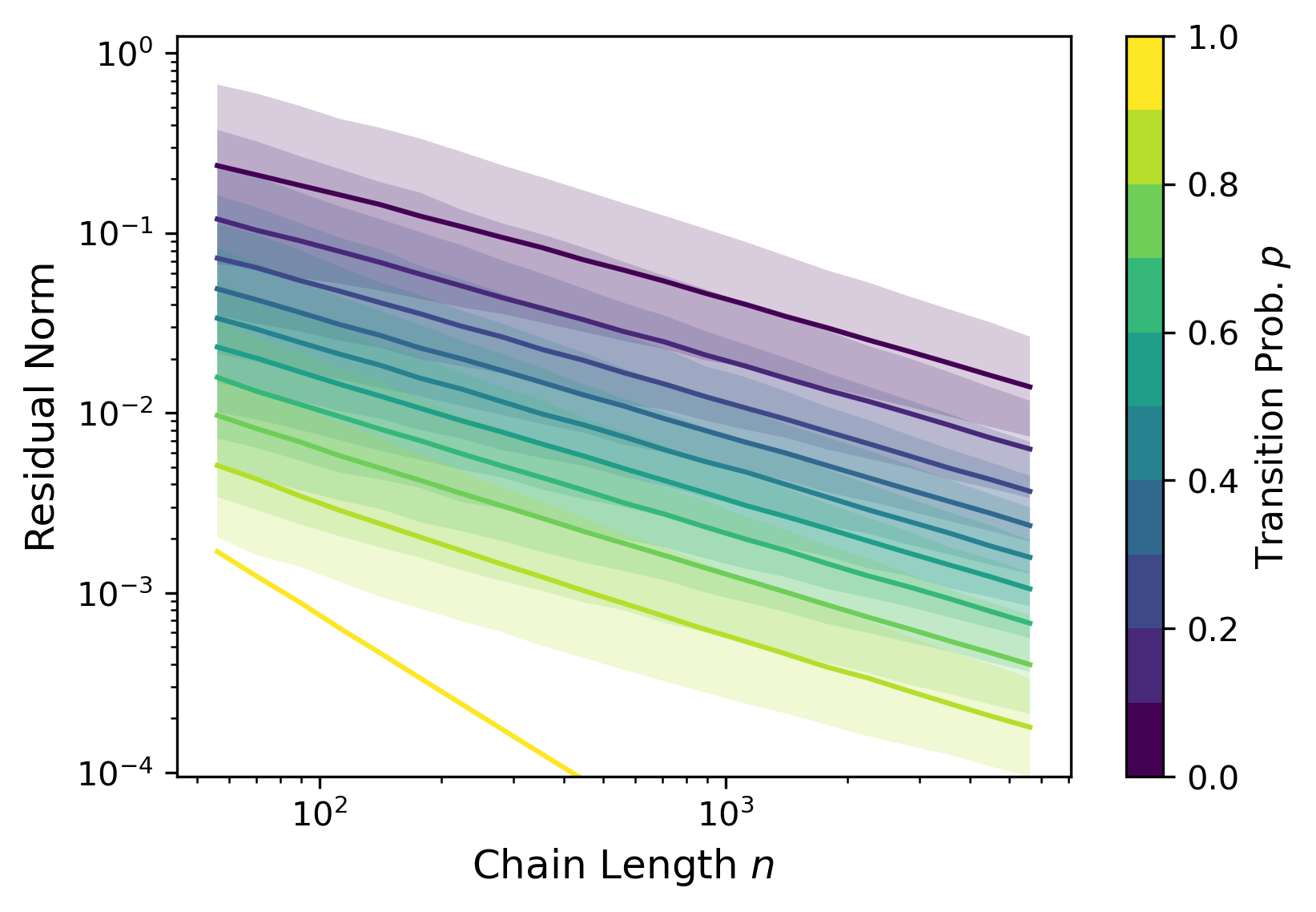}
        \caption{Sobolev-1 with $\tau = n^{-2/3}$.}
    \end{subfigure}
    \hfill
    \begin{subfigure}{0.49\linewidth}
        \centering
        \includegraphics[width=\linewidth]{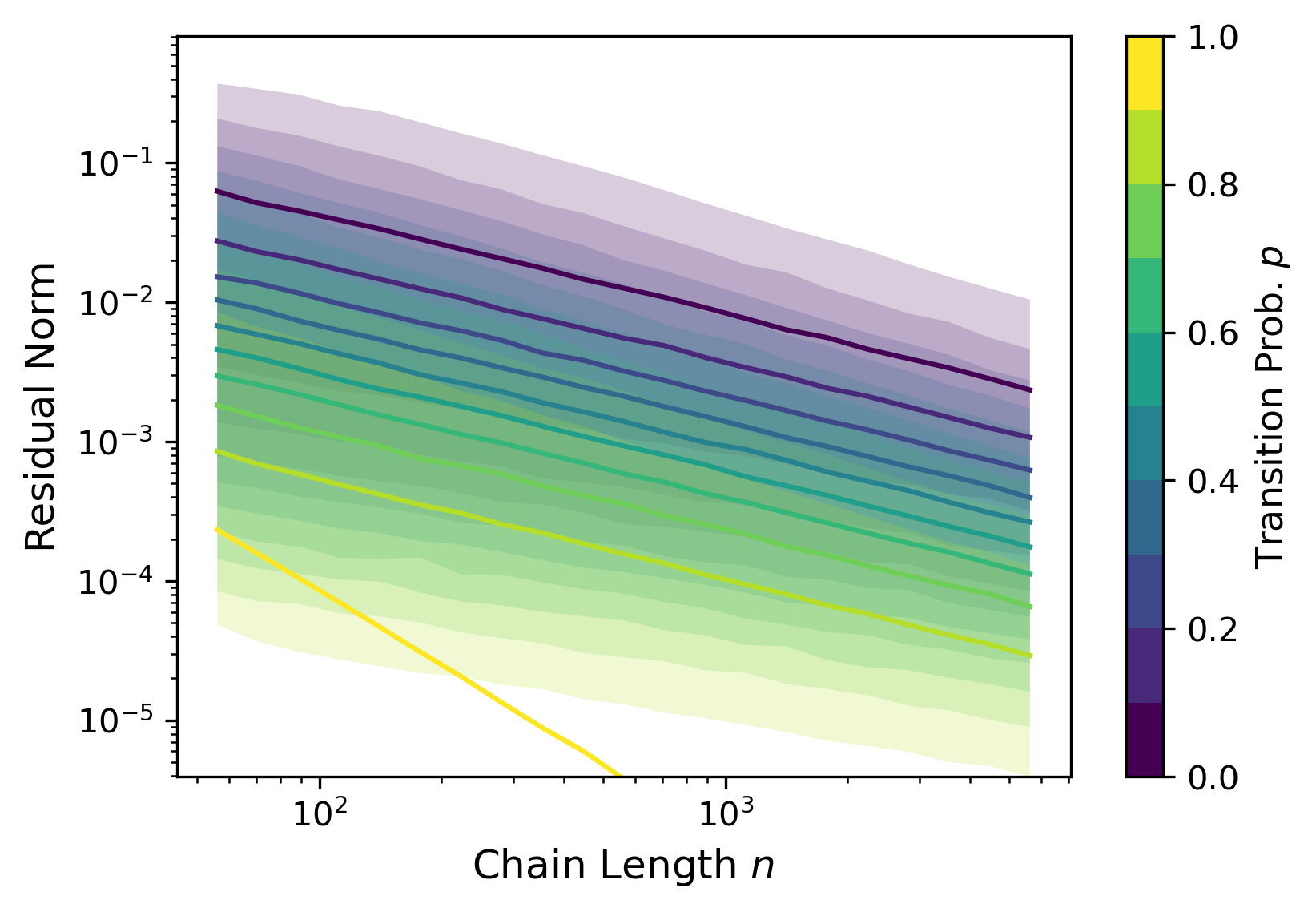}
        \caption{Sobolev-2 with $\tau = n^{-4/5}$.}
    \end{subfigure}
     \caption{Residual norm decays $\frac{1}{n}\norm{\Gamma_\tau \bm\gf}_2^2$, based on the optimal  $\tau_n$ selections, for different Sobolev-$\alpha$ kernels. Slight numerical inaccuracies are shown in the norm decay of the Sobolev-2 kernel. }%Decay in accordance with Proposition~\ref{lem:beta_lem}, although main driver of decay is tail $\beta_n$ rather than regularization $\tau_n$.}
      \label{fig:new_decays}
\end{figure}

The Sobolev-$\alpha$ RKHS on $[0,1]$ has kernel defined by inner product
\[
\langle f, g\rangle_{\Hb^\alpha} = \sum_{k=0}^\alpha \int_0^1 f^{(k)}(x)\, g^{(k)}(x)\,dx.
\]
This Mercer kernel has eigenvalue decay $\lambda_i = i^{-2\alpha}$. %In context to 
For the standard KRR problem, a minimax optimal selection of the regularization parameter %$\tau_n$ has decay rate of 
is given by $\tau_n \asymp n^{\frac{-2\alpha}{2\alpha + 1}}$. When plugged into the MSE expression, %error of KRR,
\begin{equation}
    \overline{\rm MSE}(\tau) = \frac{1}{n}\sum_{i=1}^n\frac{\tau^2 \ff_i^2}{(\lambda_i + \tau)^2} + \frac{\sigma^2}{n} \sum_{i=1}^n \frac{\lambda_i^2}{(\lambda_i + \tau)^2},
\end{equation}
we obtain $\overline{\rm MSE}(\tau_n) \asymp n^{\frac{-2\alpha}{2\alpha + 1}}$ which decays to zero as $n\to \infty$.
%Plots can be found in Figure~\ref{fig:new_decays}. 
%Decay term $\rho_n = \lambda_n /\tau_n$ goes as 
The resulting rate for $\xi_n = \lambda_n /\tau_n$ is
%$\rho_n \asymp n^{\frac{-4\alpha}{2\alpha + 1}}$ 
\[\xi_n \asymp n^{\frac{-4\alpha^2}{2\alpha + 1}} = o(1)\]
which satisfies the condition of Corollary~\ref{corr:asym_km_err}.
%for the minimax selection of $\tau_n$.
Tying back to Example~\ref{ex:1}, Figure~\ref{fig:new_decays} shows norm decay plots for $\frac{1}{n}\norm{\Gamma_\tau \bm\gf}_2^2$, %have been generated 
when using $\tau_n$ minimax selections on the different Sobolev-$\alpha$ kernels. 

Similar to the case of the min-kernel in Figure~\ref{fig:marokv}, as the signal $\bm g^*$ becomes more rough, i.e. $p$ becomes larger, we see a quicker decay in filtered norm for the different Sobolev examples. Furthermore, these contributions are filtered at a faster rate for Sobolev kernels that are smoother, that is those with larger $\alpha$ values. This faster decay is not only intuitive but expected from our earlier derived $\xi_n$ decay rate. %term.

\section{Additional Experiments}\label{app:add:exper}

\begin{table}[t]
\centering
%\scalebox{0.95}{
\begin{tabular}{ c c c l c l}
   \hline
   \multirow{2}{*}{Method} & \multirow{2}{*}{\# Seqs.} & \multicolumn{2}{c}{BrainWeb 4-class} & \multicolumn{2}{c}{BrainWeb 10-class }\\\cline{3-6}
   & & Acc. [\%] & Max Dev. [1] \rule[-0.5ex]{0ex}{2.5ex} &  Acc. [\%] & Max Dev. [1] \\\specialrule{.15em}{.05em}{.05em} 
   
   \multirow{2}{*}{\textsc{k-means}} & 1 & $73.62$ & $8.21\times 10^{-1}$\rule[-0.5ex]{0ex}{2.5ex} & $37.69$ & $7.08\times 10^{0}$ \\ \cline{2-6}
   
   & 3 & $74.38$ & $3.41\times 10^0$\rule[-0.5ex]{0ex}{2.5ex} & $44.21$ & $1.19\times 10^1$\\ \specialrule{.15em}{.05em}{.05em} 
   
   \multirow{2}{*}{\textsc{N4ITK + k-means}} & 1 & $74.10$ & $1.17\times 10^0$\rule[-0.5ex]{0ex}{2.5ex}& $40.14$ & $6.01\times 10^0$ \\ \cline{2-6}
   
   & 3 & $74.41$ & $3.91\times 10^0$\rule[-0.5ex]{0ex}{2.5ex}& $48.22$ & $1.11\times 10^1$\rule[-0.5ex]{0ex}{2.5ex} \\ \specialrule{.15em}{.05em}{.05em} 
   
   \multirow{2}{*}{\textsc{LapGM}} & 1 & $76.14$ & $2.87\times 10^0$ \rule[-0.5ex]{0ex}{2.5ex}  & $50.16$ & $\bm{2.47\times 10^0}$ \\ \cline{2-6}
   
   & 3 & $87.28$ & $4.08\times 10^0$\rule[-0.5ex]{0ex}{2.5ex}  & $78.38$ & $\bm{4.19\times 10^0}$\\ \specialrule{.15em}{.05em}{.05em} 
   
   \multirow{2}{*}{\textsc{AltMin}} & 1 & $\bm{91.07}$ &  $\bm{1.10\times 10^{-1}}$\rule[-0.5ex]{0ex}{2.5ex}  & $\bm{56.27}$ &  $4.49\times 10^{0}$\\ \cline{2-6}
   
   & 3 & $\bm{98.91}$ & $\bm{3.67\times 10^{-2}}$\rule[-0.5ex]{0ex}{2.5ex}  & $\bm{82.36}$ & $7.84\times 10^{0}$\\ \specialrule{.15em}{.05em}{.05em} 
\end{tabular}
\caption{Clustering results for different debiasing methods for single and multi-sequence settings.}\label{tab:bw_reduced}
\end{table}

We compare \altmin to other MRI debiasing techniques using the same biased phantom as Section~\ref{sec:mri_expmt}. For comparison, we consider a standard debiasing technique \textsc{N4ITK}~\citep{Tustison10} and a Bayesian modeling approach \textsc{LapGM}~\citep{Vinas22}. Hyperparameters for all methods, including \altmin, were selected using the same post-fitting process. Specific to \textsc{N4ITK}, bias estimates were calculated on the T1-sequence information and clusterings were calculated using an additional $k$-means estimation at the end of the debiasing procedure.

Performance of each method for the various recovery settings can be found in Table~\ref{tab:bw_reduced}. In each recovery setting, \altmin either meets or exceeds the classification and level accuracies of the other tested methods. We highlight that, for all debias methods, recovery is significantly more difficult in the 10-class setting. Methods which eventually scored well in this setting were those which could effectively leverage multi-sequence information during debias and clustering. This emphasizes the importance of replicated information for practical step and smooth recovery implementations.

\end{document}